\documentclass{article}
\usepackage[utf8]{inputenc}

\usepackage{geometry}
\usepackage{amsthm,mathtools,amssymb,bbm,tcolorbox}
\usepackage{cite}
\usepackage[T1]{fontenc}
\usepackage[title]{appendix}
\DeclareMathOperator{\tra}{tr}
\DeclareMathOperator{\Proj}{Proj}

\DeclareMathOperator{\vspan}{span}
\DeclareMathOperator{\var}{var}
\DeclareMathOperator{\ran}{Ran}
\DeclareMathOperator{\Ran}{Ran}

\DeclareMathOperator{\sign}{sign}

\DeclareMathOperator{\BetaDist}{Beta}

\newcommand*{\NoteFont}{\fontfamily{qcr}\selectfont}
\DeclareTextFontCommand{\textNoteFont}{\NoteFont}

\newcommand{\cl}[1]{\mathcal{#1}}
\newcommand{\bb}[1]{\mathbbm{#1}}
\newcommand{\abs}[1]{\left\lvert#1\right\rvert}

\newcommand{\tr}[1]{\tra\left[#1\right]}
\newcommand{\norm}[1]{\left\|#1\right\|}
\newcommand{\Prob}[1]{\mathbb{P}\left\{#1\right\}}
\newcommand{\EV}[2][]{\mathbb{E}_{#1}\left[#2\right]}

\newcommand{\Exp}[1]{\exp\left(#1\right)}

\newcommand{\ip}[2]{\left\langle #1, #2 \right\rangle}
\newcommand{\p}[1]{\left(#1\right)}
\newcommand{\pb}[1]{\left[#1\right]}
\newcommand{\ps}[1]{\left\{#1\right\}}

\newcommand{\given}{\vert}

\def\half{\frac{1}{2}}
\def\R{\mathbb{R}}
\def\S{\mathbb{S}}
\def\N{\mathbb{N}}
\def\C{\mathbb{C}}
\def\F{\mathbb{F}}
\def\ind{\bb{1}}

\def\eqdist{\overset{\text{(d)}}{=}}

\theoremstyle{plain}
\newtheorem{prop}{Proposition}
\newtheorem*{prop*}{Proposition}
\newtheorem{lem}[prop]{Lemma}
\newtheorem{cor}[prop]{Corollary}
\newtheorem{thm}[prop]{Theorem}
\newtheorem{definition}[prop]{Definition}

\numberwithin{prop}{subsection}

\title{Phase retrieval by random  binary questions:\\ Which complementary subspace is closer?}
\author{Dylan Domel-White and Bernhard G. Bodmann\thanks{The research in this paper was supported in part by NSF grant DMS-1715735.}}

\begin{document}

\maketitle

\begin{abstract}
Phase retrieval in real or complex Hilbert spaces is the task of recovering a vector, up to an overall unimodular multiplicative constant, from magnitudes of linear measurements.
In this paper, we assume that the vector is normalized, but retain only qualitative, binary  information
about the measured magnitudes 
by comparing them with a threshold. In more specific, geometric terms,
we choose a sequence of subspaces in a real or complex Hilbert space and only record whether a given vector is 
closer to the subspace than to the complementary subspace. 
The subspaces have half the dimension of the Hilbert space and are independent, uniformly distributed with respect to the action of the orthogonal or unitary groups.
The main goal of this paper is to find a feasible algorithm for approximate recovery based on the information gained about the vector
from these binary questions and to establish error bounds for its approximate recovery. 
We provide a pointwise bound for fixed input vectors and a uniform bound that controls the worst-case scenario
among all inputs. Both bounds hold with high probability with respect to the choice of the subspaces. 
For real or complex vectors of dimension $n$, the pointwise bound requires $m \ge C \delta^{-2} n \log(n)$ and the uniform bound
$m \ge C \delta^{-2} n^2 \log(\delta^{-1} n)$ binary questions in order to achieve  an accuracy of $\delta$. The accuracy $\delta$ is
measured by the operator norm of the difference between the rank-one orthogonal projections corresponding to the normalized input vector and its approximate recovery.
 \end{abstract}

\section{Introduction}\label{sec:Intro}

This paper is concerned with approximate phase retrieval from measuring qualitative, binary measurements. Phase retrieval is the task of recovering a vector in a real or complex Hilbert space 
up to an overall multiplicative unimodular constant from magnitudes of linear quantities.
Motivated by applications from diffraction imaging \cite{Walther:1963, Fienup:1978, Millane:1990}, or from studying properties of the Fourier transform \cite{Akutowicz:1956, Akutowicz:1957}, results on phase retrieval first focused on the case where measurements consist of magnitudes of linear functionals \cite{BalanCasazzaEtAl:2006, BandeiraCahillEtAl:2014, ConcaEdidinEtAl:2015}.
Phase retrieval with quantized measurements was studied as well \cite{PlanVershynin:2013a, MrouehRosasco:2014}, see also the preceding works \cite{BoufounosBaraniuk:2008, PlanVershynin:2013b, AiLapanowskiEtAl:2014}. In this context, quantization means the magnitudes are replaced by values from a finite alphabet. Coarse, one-bit quantization represents the extreme case, for example when only qualitative information is obtained such as how each measured magnitude compares to a single given threshold.
Another example in which coarsely quantized measurements appear is quantum state tomography, where the outcomes of experiments are recorded in order
to estimate the state of a quantum system \cite{Scott:2006,GutaKahnEtAl:2018}. In this case, the probability of an outcome is given by the squared norm of the projection of
a (normalized) state vector onto a subspace associated with the outcome. Estimating this quantum state is then, up to the known normalization, equivalent to phase retrieval for the state vector.
Phase retrieval based on norms of projections onto subspaces has also been studied outside of the context of quantum theory 
\cite{CasazzaWoodland:2014, BachocEhler:2015, CahillCasazza:2016}.
 It may be viewed as a fusion-frame version of phase retrieval, where higher rank maps replace linear functionals and the norm replaces the absolute value.
 The recovery of matrices rather than vectors is yet another higher rank generalization of phase retrieval \cite{CandesEldarEtAl:2013, PlanVershynin:2013a}.

The main goal of the present paper is to combine coarse quantization with phase retrieval from norms of projections under the assumption that the input vector $x$ is normalized.
In our setup, each measured quantity is the answer to a binary question: Is the input vector $x$ closer to a given subspace or to its orthogonal complement?
Hence, a measurement results in a binary string that encodes the orientation of $x$ in terms of the answers to the binary questions associated with a collection of subspaces.
This reduction to binary quantities is a dramatic loss of information compared to phase-insensitive, real-valued measurements.
Since the outcome of a measurement is unchanged by rescaling the input vector, we are only obtaining information about the one-dimensional subspace spanned by it.
The restriction of $x$ being a unit vector permits us to perform phase retrieval from its proximity to subspaces. 
In analogy with the unresolvable ambiguity in phase retrieval, we only seek to recover the one-dimensional subspace spanned by $x$, or equivalently, the orthogonal rank-one projection $X$ onto the span of $x$.

To achieve our goal, we use measure concentration arguments and show that measurements coming from randomly selected subspaces allow approximate recovery via a semidefinite program.
The recovery strategy in this paper can be outlined as follows: We specialize to even-dimensional real or complex Hilbert spaces and to randomized one-bit measurements based on subspaces of half the dimension.
For each random subspace $V_j$ in a sequence $\{V_1, V_2, \dots, V_m\}$, we determine whether the given input vector $x$ is closer to $V_j$ or to its orthogonal complement $V_j^\perp$.
The outcome of the binary measurement is thus encoded in a sequence of orthogonal projections $\{\hat P_1, \hat P_2, \dots, \hat P_m\}$ such that the range of each $\hat P_j$ is the subspace $\hat V_j \in \{V_j, V_j^\perp\}$ that is closest to $x$.
The answer to each binary question is equivalently determined by comparing the squared norm $\norm{P_jx}_2^2 = \tr{P_jxx^*}$ to a threshold.
For the approximate recovery of the subspace spanned by $x$ we then simply average over these orthogonal projections $\{\hat{P}_j\}_{j=1}^m$ and find the eigenspace corresponding to the largest eigenvalue of this average.
We denote the orthogonal projection onto this eigenspace by $\hat{X}$.
This operator is, in fact, the solution of a semidefinite program which maximizes $\sum_{j=1}^m\tr{\hat{P}_jY}$ in the convex set of all positive semidefinite $Y$ with $\tr{Y} \leq 1$ \cite[Section 4.2]{hornjohnson2}. This strategy is motivated by earlier results of Plan and Vershynin in the more general setting of one-bit low rank matrix recovery \cite{PlanVershynin:2013a}.

Randomized constructions and associated algorithms for recovery based on measure concentration have been studied previously in the contexts of matrix recovery, compressed sensing, and other problems in phase retrieval \cite{BoufounosBaraniuk:2008, GrossLiuEtAl:2010, JacquesLaskaEtAl:2013, CandesStrohmerEtAl:2013, PlanVershynin:2013a, PlanVershynin:2013b, AiLapanowskiEtAl:2014, BachocEhler:2015}. In contrast to the low-rank matrix recovery treated by Plan and Vershynin \cite{PlanVershynin:2013a, PlanVershynin:2013b},
we use measure concentration in operator norm to achieve our error bounds.  For a related result  based on measure concentration in operator norm but without a low-rank prior, see the work 
 by Guta and others \cite{GutaKahnEtAl:2018} on approximate quantum state tomography from measurements associated with projections onto subspaces.

In this paper, we show results that control the accuracy of the approximate recovery, in particular the decay of the error as the number of random subspaces grows.
There are two types of error estimates, pointwise and uniform in the input vector.

\emph{Pointwise Bound.} For a rank-one orthogonal projection $X$ on a real or complex $2n$-dimensional Hilbert space and a desired recovery accuracy $\delta > 0$, we show that using
\[
    m \geq C\delta^{-2}n\log(n)
\]
random subspaces for a binary measurement and the algorithm we described yields $\hat{X}$ such that the operator norm difference is bounded by $\norm{\hat{X} - X} < \delta$ with high probability. 
Here $C$ is a constant independent of $n$ and $\delta$.
See Theorem~\ref{thm:Pointwise} for the exact statement and proof of this result, along with an exact value for $C$.
One may compare this to a similar result from one-bit compressed sensing which says that $m = C\delta^{-4}n$ random one-bit measurements (of the form $X \mapsto \sign(\tr{G_jX})$ for $\{G_j\}_{j=1}^m$ independent matrices with independent standard normal entries) are sufficient to recover $\hat{X}$ with nuclear norm $\tr{\abs{\hat{X}}} = 1$ and $\tr{\abs{\hat{X}\hat{X}^*}^2} \leq 1$ such that the Hilbert-Schmidt norm $\norm{\hat{X} - X}_{HS} = \p{\tr{\abs{\hat{X} - X}^2}}^\half < \delta$ \cite[Section 3.3]{PlanVershynin:2013a}.
Another result on one-bit phase retrieval \cite{MrouehRosasco:2014} also gives comparable asymptotics when using measurements based on rank-two Gaussian random matrices.

\emph{Uniform Bound.} We also establish an error bound that holds uniformly for all rank-one projections as input with one fixed choice of subspaces for measurement. For a desired recovery accuracy $\delta > 0$, we show that using
\[
    m \geq C\delta^{-2}n^2\log(\delta^{-1}n)
\]
random subspaces for a binary measurement ensures with high probability that for each rank-one orthogonal projection $X$ we obtain $\hat{X}$ such that $\norm{\hat{X} - X} < \delta$. See Theorem~\ref{thm:Uniform} for details.

We note that for fixed $n$, the asymptotic dependence of $m$ on $\delta$ improves on results derived by Plan and Vershynin in a more general setting.
This can be attributed to our choice of measurements which are constructed with random orthogonal projections, not Gaussian matrices.
One expects that the Lipschitz regularity of the function $X \mapsto \tr{PX}$ is better than that of $X \mapsto \tr{GX}$, at least in a set of large measure among all rank-one projections.
This is advantageous, in particular in combination with perturbation arguments as in Section~\ref{subsec:SoftHammingArgument}.

We also include pointwise and uniform error bounds for faulty meaurements. In this case, up to a fixed fraction of the answers to the binary questions have been flipped, possibly in
an adversarial manner. For fixed dimension $n$, the faulty measurements contribute with an additional term in the bound for the recovery error that is proportional to the fraction
of bit flips.

The rest of this paper is organized as follows: After fixing some notation, the remainder of Section~\ref{sec:Intro} describes our one-bit phaseless measurement model in more detail; we explain how we generate random projections for each binary measurement, and how we approximately recover a signal based on such a binary measurement of it.
In Section~\ref{sec:Pointwise} we prove the error bound for our pointwise recovery, Theorem~\ref{thm:Pointwise}.
Lastly, in Section~\ref{sec:Uniform} we establish the uniform accuracy for recovery, Theorem~\ref{thm:Uniform}. Each of the main theorems in Sections~~\ref{sec:Pointwise}
and~\ref{sec:Uniform} is followed by a corollary that provides error bounds in the presence of faulty measurements. Both error bounds are also illustrated with plots showing empirical data from reconstruction using (PEP) in $\R^{16}$.

\textbf{Notation:} Since we are interested in both real and complex signals, we let $\F$ stand for either $\R$ or $\C$, and define $\beta = \half$ when $\F = \R$ and $\beta = 1$ when $\F = \C$ in order to simplify some expressions which depend on the underlying field. We consider only unit norm signals, and so denote the unit sphere in $\F^d$ by $\S_\F^{d - 1}$. As mentioned previously, both our input signals and binary measurement can be defined in terms of orthogonal projections, so we let $\Proj_\F(k,d)$ denote the space of rank-$k$ orthogonal projections on $\F^d$. For a vector $x \in \S_\F^{d-1}$, $xx^* \in \Proj_\F(1, d)$ is the rank-one projection onto the span of $x$. We write $\norm{x}$ for the euclidean norm of a vector $x \in \F^d$ and $\norm{A}$ for the operator norm of a matrix $A \in \F^{d \times d}$.

\subsection{One-Bit Phaseless Measurement Model}\label{subsec:MeasurementModel}
Our measurements are constructed from qualitative information about the proximity of $x \in \S_\F^{d-1}$ to subspaces in $\F^d$.
We formulate the measurements in terms of the orthogonal projections onto these subspaces.

For a projection $P\in \Proj_\F(k,d)$, we define its associated binary question as the map $\varphi_P: \S_\F^{d-1} \to \{0, 1\}$ given by
\begin{equation}
    \varphi_P(x) =  \begin{cases}
                        1   &\text{ if } \norm{Px}_2^2 \geq \frac{k}{d}\\
                        0  &\text{ else}.
                    \end{cases}
\end{equation}

The choice of $k/d$ as the cut-off value for quantization is natural since it is the average of $x \mapsto \norm{Px}_2^2$ over all unit vectors.
Equivalently, $k/d$ is the average of $P \mapsto \norm{Px}_2^2$ when $x$ is a fixed unit vector and $P$ is chosen uniformly at random in $\Proj_\F(k,d)$, as discussed further below in Section~\ref{subsec:RandomProjections}.

These binary questions are in fact phaseless, since $\varphi_P(x) = \varphi_P(\alpha x)$ for any $\alpha \in \F$ with $\abs{\alpha} = 1$.
Additionally, for any such $\alpha$ and any $x \in \S_\F^{d-1}$ we have $\alpha x (\alpha x)^* = xx^*$, and $\norm{Px}_2^2 = \tr{Pxx^*}$, so these binary questions can be recast as maps on the set of rank-one orthogonal projections.
In this framework --- thinking of input signals as rank-one projections --- the binary question associated to $P$ is the map $\phi_P: \Proj_\F(1, d) \to \{0,1\}$ defined by
\begin{equation}
    \phi_P(X) =  \begin{cases}
                        1   &\text{ if } \tr{PX} > \frac{k}{d}\\
                        0  &\text{ else}.
                    \end{cases}
\end{equation}
Reformulating $\varphi_P$ as $\phi_P$ encapsulates the fact that the map $\varphi_P$ is constant on the set of unit vectors that differ from $x$ by a unimodular multiplicative constant.
Henceforth, we will use this latter framework and speak of measuring and reconstructing rank-one orthogonal projections rather than unit vectors.

The binary question $\phi_P$ measures qualitative proximity information about the input signal.
For projections $P \in \Proj_\F(k, d)$ and $X \in \Proj_\F(1, d)$, $\tr{PX} = \cos^2(\theta)$, where $\theta$ is the principal angle between the one-dimensional subspace $\Ran(X)$ and the $k$-dimensional subspace $\Ran(P)$.
Thus, $\phi_P(X) = 1$ if and only if $\Ran(X)$ is closer to $\Ran(P)$ than the average for a random one-dimensional subspace, and if this occurs we say $P$ is \emph{proximal} to $X$.

Our goal is to achieve accurate phase retrieval with the qualitative proximity information gained from a sufficiently large set of these binary questions from projections  $\{P_j\}_{j=1}^m$. For such a collection, we define a corresponding binary measurement map.

\begin{definition} Given a sequence of orthogonal projections $\mathcal P=\{P_j\}_{j=1}^m$ on $\F^d$, the \emph{binary measurement map} associated with $\mathcal P$ is 
$\Phi_\cl{P}: \Proj(1,2n) \to \{0, 1\}^m$ defined by 
\begin{equation}
    \Phi_{\cl P}(X) := (\phi_{P_j}(X))_{j=1}^m\, .
\end{equation}
We also define the \emph{measurement Hamming distance} (associated with $\cl{P}$) between $X$ and $Y$ to be
\begin{equation}
    d_\cl{P}(X, Y) := d_H(\Phi_\cl{P}(X), \Phi_\cl{P}(Y))
\end{equation}
where $d_H$ denotes the normalized Hamming distance on $\{0,1\}^m$.
\end{definition}

In other words: $\Phi_\cl{P}(X)$ is a binary vector where each one-bit entry encodes the proximity of $X$ to a projection in $\cl{P}$. The value $d_\cl{P}(X,Y)$ gives the relative frequency of measurement projections that separate $X$ and $Y$, i.e. the number of binary questions in the measurement that yield different answers for $X$ and $Y$ as inputs.

\subsection{Measurement by Random Projections}\label{subsec:RandomProjections}
In the absence of an intuitive way to construct ``optimal'' collections of projections for our one-bit measurements, we instead consider projections chosen uniformly at random. The uniform probability measure on $\Proj_\F(k,d)$ is induced by the Haar measure of the unitary group $\cl{U}_\F(d)$, and is characterized by the property of being rotationally invariant, see \cite{BachocEhler:2015}. In other words, if $P$ is uniformly distributed in $\Proj_\F(k,d)$ then for any $U \in \cl{U}_\F(d)$ we have $UPU^* \eqdist P$ (where $\eqdist$ denotes equality in distribution).

In practice, there are many equivalent ways to generate a uniformly distributed rank-$k$ projection. For example, one can take $k$ Gaussian random vectors in $\F^d$ and then form the projection onto their span. A second way is to take a fixed rank-$k$ projection and conjugate it by a Haar distributed random unitary $U \in \cl{U}_\F(d)$. It can be helpful to think of a ``uniformly distributed rank-$k$ projection'' as just a ``projection onto a uniformly distributed $k$-dimensional subspace''.

For most of the paper, we work with the binary measurement map associated to a collection $\cl{P} = \{P_j\}_{j=1}^m$ of independent uniformly distributed projections in $\Proj_\F(n, 2n)$.
The reason for using half-dimensioned projections is because their associated one-bit measurements $\phi_{P}$ have a geometrically intuitive meaning: for a fixed $X \in \Proj_\F(1, 2n)$, $\phi_{P}(X) = 1$ if and only if $\tr{PX} > \frac{1}{2} \geq \tr{(I - P)X}$, i.e. the subspace $\ran(X)$ is closer to $\Ran(P)$ than to its orthogonal complement $\Ran(I - P)$.

\subsection{Approximate Phase Retrieval by Semidefinite Programming}\label{subsec:SemidefiniteProgramming}

A main goal of this paper is to use the outcomes of a random binary measurement to estimate the input accurately. 
Suppose we have measured an unknown vector $x \in \S_\F^{2n-1}$ with the binary measurement map $\Phi_{\cl P}$ associated with a random collection of projections $\cl{P} \subset \Proj_\F(n ,2n)$ and obtained the binary vector  $\Phi_\cl{P}(xx^*)$.
The information we gain from these measurements will not in general completely determine the rank-$1$ projection  $X=x x^*$ corresponding to the input vector $x$, but with enough measured quantities we can deduce a projection $\hat{X}$ which approximates $X$ in some metric.
A consistent reconstruction would seek an element $\hat X$ in the feasible set, that is, the set of all $Y$ consistent with the binary measurement in the sense that $\Phi_{\cl P} (Y) = \Phi_{\cl P}(X)$ \cite{Boyd:2004}.
A natural error bound for such a reconstruction strategy would then result from the diameter of the feasible set, which intuitively will be small if $\cl{P}$ is suitably large. 

In this paper, we relax the perfect consistency condition, but still achieve approximate recovery with a computationally feasible, semidefinite programming algorithm investigated in other works \cite[Section 4.2]{hornjohnson2}. The approximate recovery of $X$ is conveniently described in terms of projections obtained from the binary measurement $\Phi_\cl{P}(X)$.

\begin{definition}
Given $X \in \Proj_\F(1, d)$ and $P \in \Proj_\F(k, d)$ we define the \emph{proximally flipped projection}
\begin{equation}
    \hat{P}(X) := \begin{cases}
                        P     &   \text{ if } \tr{PX} \geq \frac{k}{d}\\
                        I - P &   \text{ if } \tr{PX} < \frac{k}{d}.
                    \end{cases}
\end{equation}

Next, for a sequence of orthogonal projections $\cl{P} = \{P_j\}_{j=1}^m$, the \emph{empirical average of the proximally flipped projections} is
\begin{equation}
    \hat{Q}_\cl{P}(X) := \frac{1}{m}\sum_{j=1}^m \hat{P}_j(X).
\end{equation}
\end{definition}

The recovery algorithm we study takes the binary measurement $\Phi_\cl{P}(X)$ and produces $\hat{X}$ 
that solves the semidefinite program
\begin{equation} \label{eq:SDP}\tag{PEP}
    \begin{aligned}
        & \underset{Y}{\text{maximize}}
        & & \tr{\hat{Q}_\cl{P}(X)Y} \\
        & \text{subject to}
        & & Y \succeq 0, \tr{Y} \leq 1.
    \end{aligned}
\end{equation}

We call this the Principal Eigenspace Program (PEP) because it amounts to maximizing the Rayleigh quotient \cite[Section 4.2]{hornjohnson2} for $\hat{Q}_\cl{P}(X)$. This special class of semidefinite programs can be implemented efficiently \cite[Chapter 4]{Parlett:1998}.

Since $\hat{Q}_\cl{P}(X)$ is a positive self-adjoint operator, it may be decomposed according to the spectral theorem as a linear combination of mutually orthogonal rank-1 projections $\hat{Q}_\cl{P}(X) = \sum_{i=1}^{2n}\lambda_i E_i$, where $\lambda_1 \geq \lambda_2 \geq \ldots \geq \lambda_{2n} \geq 0$. Thus, any positive self-adjoint trace normalized operator with range contained in the principal eigenspace of $\hat{Q}_\cl{P}(X)$ is a solution to (PEP). If in addition $\lambda_1$ is strictly larger than $\lambda_2$ (which happens with probability 1 for our random measurement model), then its principal eigenspace is one-dimensional, and so $\hat{X} = E_1$ is the unique solution to (PEP). Proposition \ref{prop:ExpFlippedProjs} will show that $\EV{\hat{Q}_\cl{P}(X)} = \mu_1X + \mu_2(I - X)$ with $\mu_1 > \mu_2$, and so for large $m$ we might expect $\hat{X} \approx X$ by a measure concentration argument.


Section~\ref{sec:Pointwise} of this paper shows the following pointwise result: for any fixed $X \in \Proj(1,2n)$ and any $\delta > 0$, we can choose $m$ large enough so that a collection of independent uniformly distributed half-dimensioned projections $\cl{P} = \{P_j\}_{j=1}^m$ will, with high probability, yield a measurement $\Phi_{\cl{P}}(X)$ for which the solution $\hat{X}$ to (PEP) satisfies $\norm{\hat{X} - X} < \delta$. See Theorem~\ref{thm:Pointwise} for details.

Much of the effort in Section~\ref{sec:Uniform} is directed toward getting \emph{uniform} results from the above \emph{pointwise} one. The uniform result we derive says: for any $\delta > 0$, we can choose $m$ large enough so that a collection of independent uniformly distributed half-dimensioned projections $\cl{P} = \{P_j\}_{j=1}^m$ will, with high probability, yield measurements $\Phi_\cl{P}(X)$ for every $X \in \Proj(1, 2n)$ for which the solution $\hat{X}$ to (PEP) satisfies $\norm{\hat{X} - X} < \delta$. See Theorem~\ref{thm:Uniform} for details.

According to the uniform result, we can generate a collection of projections for which \emph{every} signal is approximately recoverable up to an error of $\delta$ from the one-bit questions using those projections.
The pointwise result can be thought of as an averaged performance guarantee, whereas the uniform bound controls even the worst case input.

\section{Error bound for the approximate recovery of fixed input signals}\label{sec:Pointwise}

We begin deriving results on the statistics of signal recovery using (PEP) and our one-bit phaseless measurement model by considering a fixed unit-norm input vector $x \in \F^{2n}$ while the binary measurement map $\Phi_\cl{P}$ is chosen randomly.
As outlined before, we identify vectors that differ by a unimodular multiplicative constant, and when considering only unit-norm vectors as input signals we represent these equivalence classes by rank-one projection matrices.
The random binary measurement map is determined by a sequence of random projections $\cl{P} = \{P_j\}_{j=1}^m$ whose rank is half the dimension of the signal space, and provides information
whether the input signal is closer to the range of each projection or to its orthogonal complement. The main goal of this section is to prove that (PEP) provides accurate recovery of an input signal $X \in \Proj_\F(1, 2n)$ when sufficiently many random projections are used for the binary measurement, i.e. when $m$ is large enough.
The derivation of the results proceeds in three steps:
\begin{enumerate}
    \item[(1)] If the orthogonal projections for the measurement of $X$ are chosen uniformly at random and proximally flipped,
    then their empirical average has the expectation $Q(X) := \EV{\hat{Q}_\cl{P}(X)} = \mu_1X + \mu_2(I - X)$ where $0 < \mu_2 < \mu_1$ are constants. In particular, $X$ is the projection onto the eigenspace corresponding to the largest eigenvalue of $Q(X)$.
    \item[(2)] The empirical average $\hat{Q}_\cl{P}(X)$ concentrates near its expectation $Q(X)$.
    \item[(3)] The eigenspace of $\hat{Q}_\cl{P}(X)$ corresponding to its largest eigenvalue concentrates near $X$.
\end{enumerate}

\subsection{Expectation of $\hat{Q}_\cl{P}(X)$}\label{subsec:ExpFlippedProjs}

Before we can investigate the accuracy of (PEP), we need a simple fact about the distribution of the principal angle between a random $n$-dimensional subspace and a fixed one-dimensional subspace in $\F^{2n}$.

\begin{lem}\label{lem:DistributionOfMeas}
    Let $X \in \Proj_\F(1, 2n)$ be fixed and $P \in \Proj_\F(n, 2n)$ be uniformly distributed. Then $\tr{PX} \sim \BetaDist(\beta n, \beta n)$, i.e. $\tr{PX}$ has probability density function
    \begin{equation}
        p(t) = B\p{\beta n, \beta n}^{-1}\pb{t(1 - t)}^{\beta n - 1},
    \end{equation}
    where $B(a, b) = \int_0^1 t^{a - 1}(1 - t)^{b - 1}\ dt$ is the Beta function. In particular, $\EV{\tr{PX}} = \half$ and the distribution of $\tr{PX}$ is symmetric about $\half$.
\end{lem}
\begin{proof}
    Recall that if $U \in \cl{U}_\F(2n)$ is uniformly distributed and $E$ is the orthogonal projection onto the first $n$ standard basis vectors, then $UEU^* \overset{(d)}{=} P$. Thus
    \begin{equation*}
        \tr{PX} \overset{(d)}{=} \tr{UEU^*X} = \tr{EU^*XU}.
    \end{equation*}
    Observe that $U^*XU$ is a uniformly distributed rank-1 projection, which has the same distribution as $uu^*$ where $u \in \S_\F^{2n-1}$ is a uniformly distributed unit vector. Furthermore, $u \overset{(d)}{=} \frac{g}{\norm{g}_2}$ where $g \sim N(0,I_{2n})$ is the standard Gaussian random vector in $\F^{2n}$. So we have
    \begin{equation}\label{eq:TraceMeasurementBetaDist}
        \tr{EU^*XU} \overset{(d)}{=} \tr{Euu^*} = \norm{Eu}_2^2 \overset{(d)}{=} \frac{\norm{Eg}_2^2}{\norm{g}_2^2} = \frac{\sum_{k=1}^n \abs{g_k}^2}{\sum_{k=1}^{2n} \abs{g_k}^2}.
    \end{equation}
    
    If $\F = \R$, then the $g_k$'s are independent standard Gaussian random variables, so the right hand side of equation~(\ref{eq:TraceMeasurementBetaDist}) has the form $\frac{A}{A + B}$ where the random variables $A, B \sim \chi^2(n)$ are independent. Thus, Equation (\ref{eq:TraceMeasurementBetaDist}) is a $\BetaDist\p{\frac{n}{2}, \frac{n}{2}}$ random variable.
    
    If $\F = \C$, then each $g_k = a_k + ib_k$ where all the $a_k$ and $b_k$'s are independent standard random variables. In this case, since $\abs{g_k}^2 = \abs{a_k}^2 + \abs{b_k}^2$, the right hand side of Equation \ref{eq:TraceMeasurementBetaDist} has the form $\frac{A}{A + B}$ where $A, B \sim \chi^2(2n)$ are independent, and thus is a $\BetaDist(n,n)$ random variable.
\end{proof}

Next we compute the expectation of the empirical average of the proximally flipped projections.

\begin{prop}\label{prop:ExpFlippedProjs}
    Let $X \in \Proj_\F(1, 2n)$ and $\cl{P} = \{P_j\}_{j=1}^m$ be an independent sequence of uniformly distributed projections in $\Proj_\F(n,2n)$, then
    \begin{equation}
        Q(X) = \mu_1X + \mu_2(I - X),
    \end{equation}
    where
    \begin{equation}
        \begin{split}
            \mu_1 = \half + \frac{1}{\beta n 4^{\beta n} B\p{\beta n, \beta n}},
        \end{split}
        \quad\quad
        \begin{split}
            \mu_2 = \half - \frac{1}{\beta n (2 n - 1) 4^{\beta n} B\p{\beta n, \beta n}}.
        \end{split}
    \end{equation}
\end{prop}
\begin{proof}
    We begin with some manipulation and reasoning that does not depend on whether $\F$ is $\R$ or $\C$, which only makes a difference when computing the values of $\mu_1$ and $\mu_2$.
    
    Since the $P_j$'s are identically distributed, we know that $\EV{\hat{P}_i(X)} = \EV{\hat{P}_j(X)}$ for all $i$ and $j$. Thus, by linearity of expectation we have $Q(X) = \EV{\hat{P}_1(X)}$.
    
    Also, the distribution of $\hat{P}_1(X)$ is invariant under conjugation with a unitary that fixes $X$. In other words, for a unitary $U \in \cl{U}_\F(2n)$ such that $UXU^* = X$, then $U\hat{P}_1(X)U^* \overset{(d)}{=} \hat{P}_1(X)$. To verify this, we use the rotational invariance of $P_1$ and the cyclic property of the trace to obtain
    \begin{equation*}
       U \hat P_{1}(X) U^* \eqdist U \widehat{(U^* P_1 U)}(X)\, U^* = \hat P_{1}( U X U^*) = \hat P_1(X) . 
    \end{equation*}
    
    Using the linearity of expectation once more, it follows that $Q(X)$ is also invariant under conjugation by unitaries that fix $X$. 
   This implies that every eigenspace of $Q(X)$ is preserved under rotations by all such unitaries, hence $\ran(X)$ and $\ran(X)^\perp$ are the eigenspaces of $Q(X)$. Letting $\mu_1$ and $\mu_2$ denote the respective eigenvalues, we write
    \begin{equation}
        Q(X) = \mu_1 X + \mu_2 (I - X).
    \end{equation}

    In order to determine the value of $\mu_1$, we use linearity of expectation to see
    \begin{equation}
        \mu_1 = \tr{Q(X)X} = \EV{\tr{\hat{P}_1(X)X}}.
    \end{equation}
    By the law of total probability we have,
    \begin{align*}
        \EV{\tr{\hat{P}_1(X)X}} &= \EV{\tr{\hat{P}_1(X)X}\ \biggr\vert\ \tr{P_1X} \geq \half}\Prob{\tr{P_1X} \geq \half}\\
        &+ \EV{\tr{\hat{P}_1(X)X}\ \biggr\vert\ \tr{P_1X} < \half}\Prob{\tr{P_1X} < \half}, \nonumber
    \end{align*}
    so by the definition of $\hat{P}_1(X)$ and the symmetry of the distribution of $\tr{P_1X}$ --- a consequence of Lemma~\ref{lem:DistributionOfMeas} --- it follows that
    \begin{equation}
        \EV{\tr{\hat{P}_1(X)X}} = \EV{\tr{PX} \biggr\vert\ \tr{PX} \geq \half}.
    \end{equation}
    
    We can compute this conditional expectation using integration by parts with the probability density function given in Lemma~\ref{lem:DistributionOfMeas}, yielding
    \begin{equation}
        \mu_1 = 2B\p{\beta n, \beta n}^{-1}\int_\half^1 \pb{t(1-t)}^{\beta n - 1}\ dt = \half + \frac{1}{\beta n 4^{\beta n}\ B\p{\beta n, \beta n}}.
    \end{equation}
    Since $\tr{Q(X)} = n$ by linearity of expectation, we know $\mu_1 + (2n-1)\mu_2 = n$, from which we get the desired expression for $\mu_2$.
    
\end{proof}

\subsection{Concentration of $\hat{Q}_\cl{P}(X)$ near $Q(X)$}\label{Pointwise-ConcentrationOfQ}

Since the empirical average of the proximally flipped projections $\hat{Q}_\cl{P}(X)$ is, after all, an empirical average, by the law of large numbers it should concentrate tightly around its expectation $Q(X)$ as the number of measurements $m$ goes to infinity. To make this precise, we use the Matrix Bernstein Inequality \cite[Theorem 1.6.2]{Tropp:2015}.

\begin{lem}\label{lem:EmpiricalAverageConcentration}
    Let $X \in \Proj_\F(1, 2n)$ and $\cl{P} = \{P_j\}_{j=1}^m$ be an independent sequence of uniformly distributed projections in $\Proj_\F(n, 2n)$. Then
    \begin{equation}
        \EV{\norm{\hat{Q}_\cl{P}(X) - Q(X)}} \leq \sqrt{\frac{\log(4n)}{2m}} + \frac{\log(4n)}{3m},
    \end{equation}
    and for any $t > 0$,
    \begin{equation}
        \Prob{\norm{\hat{Q}_\cl{P}(X) - Q(X)} \geq t} \leq 4n\Exp{-\frac{6t^2m}{7}}.
    \end{equation}
    
    In particular, if $m \geq \frac{7}{6}t^{-2}(\log(4n) + D)$ then
    \begin{equation}\label{eq:ConcEmpAvgFlipped}
        \Prob{\norm{\hat{Q}_\cl{P}(X) - Q(X)} \geq t} \leq \Exp{-D}.
    \end{equation}
\end{lem}
\begin{proof}
    Let $S_j = \frac{1}{m}(\hat{P}_j(X) - Q(X))$. Then $\EV{S_j} = 0$ and $\norm{S_j} \leq \frac{1}{m}$ for all $j = 1, \ldots, m$. Note that $Z := \sum_{j=1}^m S_j = \hat{Q}_\cl{P}(X) - Q(X)$. Additionally, since $\hat{P}_j(X)$ is a projection and $\EV{\hat{P}_j(X)} = Q(X)$ for all $j$, we may bound the matrix variance
    \begin{equation}
        v(Z) := \norm{\sum_{j=1}^m \EV{S_j^2}} = \frac{1}{m}\norm{Q(X) - Q(X)^2} \leq \frac{1}{4m}.
    \end{equation}
    
    The expectation bound and tail bound now follow from applying the Matrix Bernstein Inequality as in \cite[Theorem 1.6.2]{Tropp:2015}. Additionally, if $m \geq \frac{7}{6}t^{-2}(\log(4n) + D)$ then $\log(4n) - \frac{6t^2m}{7} \leq  -D$, which yields (\ref{eq:ConcEmpAvgFlipped}).
\end{proof}

\subsection{Concentration of $\hat{X}$ near $X$ (Pointwise Result)}
From Lemma~\ref{lem:EmpiricalAverageConcentration} we know that, with enough measurement projections, with high probability $\hat{Q}_\cl{P}(X)$ is close to $Q(X)$ in operator norm. When it is sufficiently close, then the eigenspace of $\hat{Q}_\cl{P}(X)$ corresponding to its maximum eigenvalue will also be close to $X$. To see this, we first need the following lemma. It is based on the fact that for two rank-one projections $X$ and $Y$, the difference $X - Y$ is a zero-trace self-adjoint operator of rank two, and hence has a spectral representation of the form $X - Y = \norm{X - Y}(A - B)$ with two mutually orthogonal rank-one projections $A$ and $B$.

\begin{lem}\label{lem:ExpectedDifferenceTraceMeasurements}
    Let $X, Y \in \Proj_\F(1, 2n)$. Then
    \begin{equation}
        \norm{X - Y} = (\mu_1 - \mu_2)^{-1}\tr{Q(X)(A - B)}
    \end{equation}
    where $A, B \in \Proj_\F(1, 2n)$ are the mutually orthogonal projections in the spectral decomposition $X - Y = \norm{X - Y}(A - B)$.
\end{lem}
\begin{proof}
    Let $\theta$ be the principal angle between the subspaces associated to $X$ and $Y$. Then we can pick $x, y, z \in \S_\F^{2n-1}$ with $x \perp z$ such that $X = xx^*, Y = yy^*$ and $y = \cos(\theta)x + \sin(\theta)z$. Then
    \begin{equation*}
        Y = yy^* = \cos^2(\theta)xx^* + \sin^2(\theta)zz^* + \sin(\theta)\cos(\theta)(xz^* + zx^*).
    \end{equation*}
    Since $Q(X) = \mu_1X + \mu_2(I - X)$,
    \begin{equation*}
        \tr{Q(X)(X-Y)} = \mu_1 - \cos^2(\theta)\mu_1 - \sin^2(\theta)\mu_2 = (\mu_1 - \mu_2)\sin^2(\theta).
    \end{equation*}
    Lastly, since $\sin(\theta) = \norm{X - Y}$, rewriting the left hand side using the spectral decomposition $X - Y = \norm{X - Y}(A - B)$ and cancelling the common factor of $\norm{X - Y}$ yields the desired equality.
\end{proof}

The spectral gap $\mu_1 - \mu_2$ of $Q(X)$ appears in the sufficient number of binary questions in both our pointwise and uniform result. The following lemma bounds this quantity in in terms of the dimension $n$.

\begin{lem}\label{lem:SpecGapBound}
    Let $\mu_1$ and $\mu_2$ be as in Proposition~\ref{prop:ExpFlippedProjs}. Then for $\beta n \geq 2$ 
    \begin{equation}
       \frac{(n-1)\sqrt{2\beta n - 1}}{\sqrt{2\pi}\beta n (2n - 1)} \leq \mu_1 - \mu_2 \leq \frac{4(n-1)\sqrt{2\beta n - 1}}{e\sqrt{2\pi}\beta n (2n - 1)}
    \end{equation}
    In particular, $(\mu_1 - \mu_2)^{-1} = O(\sqrt{n})$.
    \begin{proof}
        From the expressions derived in Proposition~\ref{prop:ExpFlippedProjs} we have
        \begin{equation}\label{eq:SpecGapExpression}
            \mu_1 - \mu_2 = \frac{2(n-1)}{\beta n (2n - 1) 4^{\beta n}B(\beta n, \beta n)}.
        \end{equation}
        Since $B(\alpha_1, \alpha_2) = \frac{\Gamma(\alpha_1)\Gamma(\alpha_2)}{\Gamma(\alpha_1 + \alpha_2)}$, we may use Stirling's formula to approximate the Beta function. In particular, from \cite{Robbins:1955} we have for all real numbers $k \geq 2$
        \begin{equation}
            \Exp{\frac{1}{12k - 11}} \leq \frac{\Gamma(k)}{\sqrt{2\pi}(k - 1)^{k - \half}\Exp{-(k - 1)}} \leq \Exp{\frac{1}{12k - 12}}.
        \end{equation}
        In particular, when $\beta n \geq 2$ these inequalities for the Gamma function yield the bounds
        \begin{equation}\label{eq:BetaBound}
            \frac{e\sqrt{2\pi}}{2 \cdot 4^{\beta n}\sqrt{2\beta n - 1}} \leq  B(\beta n, \beta n) \leq \frac{2\sqrt{2\pi}}{4^{\beta n} \sqrt{2\beta n - 1}}.
        \end{equation}
              Using these bounds for the Beta function in (\ref{eq:SpecGapExpression}) gives the desired inequalities for $\mu_1 - \mu_2$.
    \end{proof}
\end{lem}

Now we have the tools to prove the pointwise error bound for approximate recovery of a fixed input signal using (PEP).

\begin{thm}\label{thm:Pointwise}
    Let $X \in \Proj_\F(1, 2n)$ and $\delta > 0$ be fixed. If
    \begin{equation}
        m \geq \frac{14}{3}(\mu_1 - \mu_2)^{-2}\delta^{-2}(\log(4n) + D),
    \end{equation}
    and $\cl{P} = \{P_j\}_{j=1}^m$ is an independent sequence of uniformly distributed projections in $\Proj_\F(n, 2n)$, then with probability at least $1 - \Exp{-D}$
    \begin{equation}\label{eq:mPointwise}
        \norm{\hat{X} - X} < \delta,
    \end{equation}
    where $\hat{X}$ is the solution to (PEP) with input $\Phi_\cl{P}(X)$.
\end{thm}
\begin{proof}
    From Lemma~\ref{lem:ExpectedDifferenceTraceMeasurements}, we know that
    \begin{equation}
        \norm{\hat{X} - X} = (\mu_1 - \mu_2)^{-1}\tr{Q(X)(A - B)},
    \end{equation}
    where $A,B \in \Proj_\F(1, 2n)$ are the orthogonal projections from the spectral decomposition of the difference $X - \hat{X} = \norm{X - \hat{X}}(A - B)$.
    
    Since $\hat{X}$ is the projection onto the principal eigenspace of $\hat{Q}_\cl{P}(X)$, we see
    \begin{equation}
        \tr{\hat{Q}_\cl{P}(X)(\hat{X} - X)} \geq 0 \implies (\mu_1 - \mu_2)^{-1}\tr{\hat{Q}_\cl{P}(X)(B - A)} \geq 0,
    \end{equation}
    and so
    \begin{equation}\label{eq:ErrorOpNormIneq}
        \norm{\hat{X} - X} \leq (\mu_1 - \mu_2)^{-1}\tr{(Q(X) - \hat{Q}(X))(A - B)} \leq 2(\mu_1 - \mu_2)^{-1}\norm{Q(X) - \hat{Q}(X)}.
    \end{equation}
    
    We have chosen $m$ such that $m \geq \frac{7}{6}t^{-2}(\log(4n) + D)$ for $t = \half (\mu_1 - \mu_2)\delta$, so the tail bound in Lemma~\ref{lem:EmpiricalAverageConcentration} says with probability at least $1 - \Exp{-D}$ we have $\norm{\hat{Q}_\cl{P}(X) - Q(X)} < t$. If this occurs, then from (\ref{eq:ErrorOpNormIneq}) we see
    \begin{equation}
        \norm{\hat{X} - X} \leq 2(\mu_1 - \mu_2)^{-1}\norm{\hat{Q}_\cl{P}(X) - Q(X)} < \delta.
    \end{equation}
\end{proof}

\begin{figure}
    \centering
    \includegraphics[width=0.8\textwidth]{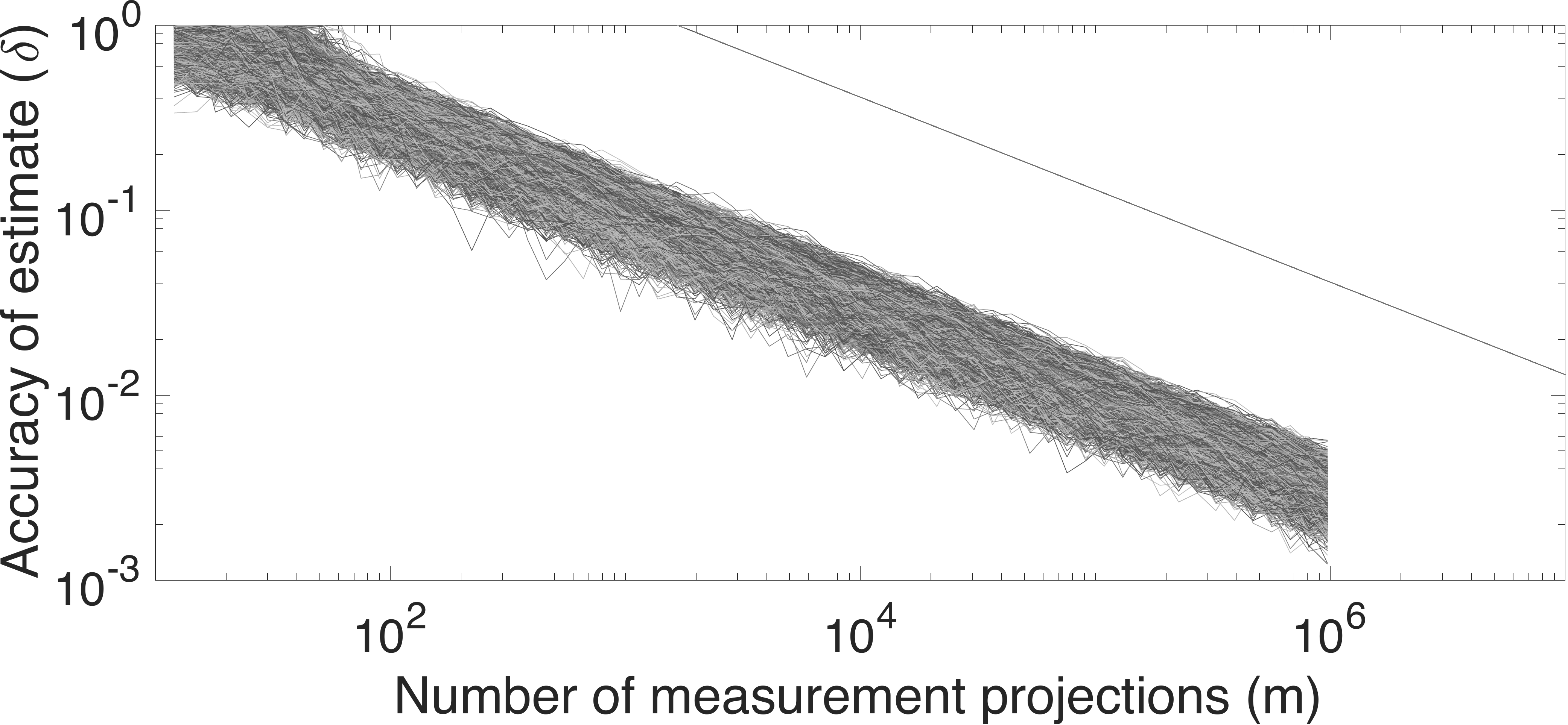}
    \caption{Plot showing the accuracy of recovery using (PEP) for a fixed input and 7200 independent collections of up to $10^6$ measurement projections on $\R^{16}$. The single line separate from the cluster represents the upper bound on $\delta$ given by Theorem~\ref{thm:Pointwise}.}
    \label{fig:PointwisePlot}
\end{figure}

See Figure~\ref{fig:PointwisePlot} for a plot showing how our bound on the sufficient number of measurements to achieve an accuracy of $\delta$ relates to experimental results.

Our proof lets us fine tune the probability of successful recovery by adjusting the value of $D$ in (\ref{eq:mPointwise}). By increasing $D$ we increase the probability of success, but also increase the sufficient number of measurements. In particular, we can take $D = \alpha\log(n)$ to ensure success with high probability, i.e. the failure rate decays on the order of $n^{-\alpha}$. To do so, we gain a constant factor that depends on $\alpha$ in the number of sufficient measurement projections $m$.

\begin{cor}\label{cor:PointwiseHighProb}
     Let $X \in \Proj_\F(1, 2n)$ and $\delta > 0$ be fixed. If $\alpha > 0$, $m \geq C_\alpha\delta^{-2}n\log(n)$, and $\cl{P} = \{P_j\}_{j=1}^m$ is an independent sequence of uniformly distributed projections in $\Proj_\F(n, 2n)$, then with probability at least $1 - n^{-\alpha}$
    \begin{equation}
        \norm{\hat{X} - X} < \delta,
    \end{equation}
    where $\hat{X}$ is the solution to (PEP) with input $\Phi_\cl{P}(X)$ and $C$ is a constant that only depends on $\alpha$.
\end{cor}

We can also take $D = n$ in (\ref{eq:mPointwise}) to ensure success with overwhelming probability, i.e. the failure rate decays on the order of $\Exp{-n}$. In this case, we gain an additional factor of $n$ in the number of sufficient measurement projections $m$.

\begin{cor}\label{cor:PointwiseOverwhelming}
     Let $X \in \Proj_\F(1, 2n)$ and $\delta > 0$ be fixed. If $m \geq C\delta^{-2}n^2\log(n)$ and $\cl{P} = \{P_j\}_{j=1}^m$ is an independent sequence of uniformly distributed projections in $\Proj_\F(n, 2n)$, then with probability at least $1 - \Exp{-n}$
    \begin{equation}
        \norm{\hat{X} - X} < \delta,
    \end{equation}
    where $\hat{X}$ is the solution to (PEP) with input $\Phi_\cl{P}(X)$ and $C$ is a constant.
\end{cor}

The pointwise accuracy guarantee of Theorem~\ref{thm:Pointwise} can also be thought of as an ``average case'' error bound with respect to the random sequence of measurement projections $\cl{P}$. The following corollary makes this explicit.

\begin{cor}
    Let $X \in \Proj_\F(1, 2n)$ and $\delta > 0$ be fixed. If $m \geq C\delta^{-2}n\log(n)$ and $\cl{P} = \{P_j\}_{j=1}^m$ is an independent sequence of uniformly distributed projections in $\Proj_\F(n, 2n)$, then 
    \begin{equation}
        \EV{\norm{\hat{X} - X}} < \delta,
    \end{equation}
    where $\hat{X}$ is the solution to (PEP) with input $\Phi_\cl{P}(X)$ and $C$ is a constant.
\end{cor}
\begin{proof}
    Take the expectation on both sides of (\ref{eq:ErrorOpNormIneq}) with respect to the random sequence of projections $\cl{P}$ and use the expectation bound from Lemma~\ref{lem:EmpiricalAverageConcentration}.
\end{proof}

We also remark that recovery using (PEP) is robust to bit-flip errors in the binary measurement, which can be seen via a small addition to the proof of Theorem~\ref{thm:Pointwise}. 
To this end, we consider a faulty measurement $\tilde{\Phi}_{\cl P}$ with the property that the normalized Hamming distance between the faulty and correct measurements
is bounded by a fixed fraction
$d_H(\tilde{\Phi}_{\cl P}(X), {\Phi}_\cl{P}(X)) \leq \tau$.

\begin{cor}
    Let $X$, $\delta$, $m$, and $\{P_j\}_{j=1}^m$ be as in Theorem~\ref{thm:Pointwise}, and fix $0 < \tau < 1$. Then with probability at least $1 - \Exp{-D}$, for all $\tilde{\Phi}_\cl{P}(X) \in \{0,1\}^m$ such that
    \begin{equation}
        d_H(\Phi_\cl{P}(X), \tilde{\Phi}_\cl{P}(X)) \leq \tau,
    \end{equation}
    we have
    \begin{equation}
        \norm{\tilde{X} - X} \leq \delta + 2(\mu_1 - \mu_2)^{-1}\tau,
    \end{equation}
    where $\tilde{X}$ denote the solution to (PEP) with input $\tilde{\Phi}_\cl{P}(X)$
    and $\mu_1-\mu_2$ is controlled by Lemma~\ref{lem:SpecGapBound}. 
    
\end{cor}
\begin{proof}
    Let $\tilde{Q}_\cl{P}(X)$ denote the empirical average of the (faulty) flipped projections, i.e. flipped using $\tilde{\Phi}_\cl{P}(X)$ rather than $\Phi_\cl{P}(X)$. Then as before, we have
    \begin{equation}
        \norm{\tilde{X} - X} \leq 2(\mu_1 - \mu_2)^{-1}\norm{\tilde{Q}_\cl{P}(X) - Q(X)}.
    \end{equation}
    By the triangle inequality, it follows that
    \begin{equation}
        \norm{\tilde{Q}_\cl{P}(X) - Q(X)} \leq \norm{\tilde{Q}_\cl{P}(X) - \hat{Q}_\cl{P}(X)} + \norm{\hat{Q}_\cl{P}(X) - Q(X)}.
    \end{equation}
    Since the normalized Hamming distance between $\Phi_\cl{P}(X)$ and $\tilde{\Phi}_\cl{P}(X)$ is bounded by $\tau$, we see
    \begin{equation}
        \norm{\tilde{Q}_\cl{P}(X) - \hat{Q}_\cl{P}(X)} \leq \tau.
    \end{equation}
    Since $\norm{\hat{Q}_\cl{P}(X) - Q(X)} \leq \delta$ with probability at least $1 - \Exp{-D}$ by the same proof as in Theorem~\ref{thm:Pointwise}, the result follows.
\end{proof}

We expect that a deeper analysis will reveal a better dependence on the error rate, or perhaps eliminate the dimension dependent factor $(\mu_1-\mu_2)^{-1}$.

\section{From pointwise to uniformly accurate recovery}\label{sec:Uniform}

In this section we extend the result from Theorem \ref{thm:Pointwise} to show that the recovery error using (PEP) is small uniformly across all input vectors $X \in \Proj_\F(1,2n)$ for a single random binary measurement $\Phi_\cl{P}$. Our strategy consists of the following steps:
\begin{itemize}
    \item[(1)] Using sufficiently many random projections, $\hat{Q}_\cl{P}(X)$ concentrates near $Q(X)$ for all $X$ in an $\epsilon$-net of $\Proj_\F(1, 2n)$.
    \item[(2)] With high probability the measurement Hamming distance between a pair $X, Y \in \Proj_\F(1, 2n)$ is not much larger than $\norm{X - Y}$, uniformly for all such pairs.
    \item[(3)] The eigenspace of $\hat{Q}_\cl{P}(X)$ corresponding to its largest eigenvalue concentrates near $X$ uniformly for all $X \in \Proj_\F(1, 2n)$.
\end{itemize}

\subsection{Concentration of $\hat{Q}_\cl{P}(X)$ near $Q(X)$ uniformly on a net}\label{subsec:ConcEmpAvgNet}

First, we show an inequality relating the Euclidean distance between unit vectors to the operator norm distance between their associated rank-one projections.

\begin{lem}\label{lem:NormInequality}
    Let $d \in \N$. Then for all $x, y \in \S_\F^{d-1}$,
    \begin{equation}
        \norm{xx^* - yy^*} \leq \norm{x - y}_2.
    \end{equation}
\end{lem}
\begin{proof}
    Let $\theta$ be the principal angle between the subspaces associated to $xx^*$ and $yy^*$, and recall $\norm{xx^* - yy^*} = \sin(\theta)$. Thus
    \begin{equation*}
        \norm{x - y}_2^2 = \ip{x - y}{x - y} = 2 - 2\Re{\ip{x}{y}} \geq 2 - 2\abs{\ip{x}{y}} = 2 - 2\cos(\theta).
    \end{equation*}
    Since $\theta \in [0, \frac{\pi}{2}]$ we know $0 \leq \cos(\theta) \leq 1$ and so
    \begin{equation*}
        2 - 2\cos(\theta) = 2(1 - \cos) \geq (1 + \cos(\theta))(1 - \cos(\theta)) = \sin^2(\theta) = \norm{xx^* - yy^*}^2.
    \end{equation*}
\end{proof}

Next, we use Lemma~\ref{lem:NormInequality} to prove the existence of $\epsilon$-nets of $\Proj_\F(1, 2n)$ with explicit cardinality bounds. This follows from the analogous results for $\epsilon$-nets of $\S_\F^{2n-1}$.

\begin{lem}\label{lem:ProjNetSize}
    For any $\epsilon > 0$, there exists an $\epsilon$-net $\cl{N}_\epsilon$ for $\Proj_\F(1, 2n)$ with respect to the operator norm with cardinality satisfying
    \begin{equation}
        \log\abs{\cl{N}_\epsilon} \leq 4\beta n\log(1 + 2\epsilon^{-1}).
    \end{equation}
\end{lem}
\begin{proof}
    By the standard volume bound for the covering number of the sphere in real euclidean space \cite{BaraniukDavenportEtAl:2008}, and the fact that $\S_\C^{2n-1}$ is naturally isometric to $\S_\R^{4n-1}$, for every $\epsilon > 0$ there exists an $\epsilon$-net $\cl{N}'_\epsilon$ for $\S_\F^{2n-1}$ (with respect to the Euclidean distance) with cardinality satisfying
    \begin{equation*}
        \abs{\cl{N}'_\epsilon} \leq \left(1 + \frac{2}{\epsilon}\right)^{4\beta n}.
    \end{equation*}
    By Lemma~\ref{lem:NormInequality}, $\cl{N}_\epsilon := \{xx^*: x \in \cl{N}_\epsilon'\}$ is an $\epsilon$-net for $\Proj_\F(1, 2n)$ with the desired cardinality bound.
\end{proof}

Now that we have existence of epsilon-nets with control on their cardinality, we use a union bound and Lemma~\ref{lem:EmpiricalAverageConcentration} to show that with sufficiently many measurements, $\hat{Q}_\cl{P}(X)$ concentrates near $Q(X)$ uniformly for all $X$ in an epsilon-net of $\Proj_\F(1, 2n)$.

\begin{lem}\label{lem:EmpAvgConcNet}
    Let $\epsilon > 0$ and $\cl{N}_\epsilon$ be an $\epsilon$-net of $\Proj_\F(1,2n)$ such that $\log\abs{\cl{N}_\epsilon} \leq 4\beta n\log(1 + 2\epsilon^{-1})$. Also, let $\delta > 0$, $m \geq \frac{7}{6}\delta^{-2}\pb{\log(4n) + 4\beta n\log(1 + 2\epsilon^{-1}) + D}$, and $\cl{P} = \{P_j\}_{j=1}^m$ be an independent sequence of uniformly distributed projections in $\Proj_\F(n,2n)$. Then with probability at least $1 - \Exp{-D}$ we have
    \begin{equation}
        \norm{\hat{Q}_\cl{P}(X) - Q(X)} \leq \delta
    \end{equation}
    for all $X \in \cl{N}_\epsilon$.
\end{lem}
\begin{proof}
    By Lemma~\ref{lem:EmpiricalAverageConcentration} and our assumption on $m$, for each $X \in \cl{N}_\epsilon$ we know
    \begin{equation*}
        \Prob{\norm{\hat{Q}_\cl{P}(X) - Q(X)} \geq \delta} \leq \Exp{-4\beta n\log(1 + 2\epsilon^{-1}) - D}.
    \end{equation*}
    By taking a union bound over all $X \in \cl{N}_\epsilon$ it follows that
    \begin{equation*}
        \Prob{\norm{\hat{Q}_\cl{P}(X) - Q(X)} \leq t \text{ for all } X \in \cl{N}_\epsilon} \geq 1 - \abs{\cl{N}_\epsilon}\Exp{-4\beta n\log(1 + 2\epsilon^{-1}) - D}.
    \end{equation*}
    The claim follows from our upper bound on $\abs{\cl{N}_\epsilon}$.
\end{proof}

\subsection{Relation between the measurement Hamming distance and operator norm distance}\label{subsec:SoftHammingArgument}

The main goal of this section is to prove our guarantee for uniformly accurate recovery, Theorem~\ref{thm:UnifHammingOpIneq}: With sufficiently many measurements, with high probability the measurement Hamming distance between any pair $X,Y \in \Proj_\F(1, 2n)$ is not much larger than the operator norm of their difference. 
It is relatively simple to show that this happens for fixed $X$ and $Y$, but showing that it holds uniformly for all such pairs requires more complicated techniques. 
To this end, we will define the \emph{$t$-soft Hamming distance} similarly as in Plan and Vershynin's \emph{Dimension reduction by random hyperplane tessellations} \cite{PlanVershynin:2014}.
We establish a continuity property and concentration results for the $t$-soft Hamming distance, which allow us to show uniform concentration of the measurement Hamming distance near its expected value over all of $\Proj_\F(1, 2n)$.
We then show that $\EV{d_\cl{P}(X,Y)}$ can be bounded in terms of $\norm{X - Y}$, after which Theorem $\ref{thm:UnifHammingOpIneq}$ follows.

\subsubsection{The $t$-soft Hamming distance and its continuity properties}
For any $X,Y \in \Proj_\F(1, 2n)$ let $\cl{S}_{X,Y} := \{P \in \Proj_\F(n, 2n): \phi_P(X) \neq \phi_P(Y)\}$, i.e. the set of projections that yield different measurements of $X$ and $Y$.
If $P \in \cl{S}_{X,Y}$, then we say that $P$ \emph{separates} $X$ and $Y$.
For a sequence $\cl{P} = \{P_j\}_{j=1}^m \subset \Proj_\F(n, 2n)$, notice that $d_\cl{P}(X, Y) = \frac{1}{m}\sum_{j=1}^m \ind_{\cl{S}_{X,Y}}(P_j)$.

With this expression for the measurement Hamming distance in mind, we define
\begin{align}
    \cl{S}^t_{X,Y} :=\ &\{P \in \Proj_\F(n, 2n): \tr{PX} + t < \half \leq \tr{PY} - t\}\\
    &\cup \{P \in \Proj_\F(n, 2n): \tr{PY} + t < \half \leq \tr{PX} - t\}\nonumber
\end{align}
for all $t \in \R$, and if $P \in \cl{S}_{X,Y}^t$ then we say $P$ \emph{$t$-separates} $X$ and $Y$. 

\begin{definition}
Given a sequence of orthogonal projections $\cl{P}=\{P_j\}_{j=1}^m$ in $\Proj_\F(n, 2n)$ and $t \in \R$, we define the \emph{$t$-soft Hamming distance} between input projections $X,Y \in \Proj_\F(1, 2n)$ to be
\begin{equation}
    d_\cl{P}^t(X,Y) := \frac{1}{m}\sum_{j=1}^m \ind_{\cl{S}_{X,Y}^t}(P_j).
\end{equation}
\end{definition}


Ultimately we want to prove uniform results for the measurement Hamming distance, but its discontinuity causes problems with standard $\epsilon$-net arguments. The $t$-soft Hamming distance helps us work around this discontinuity, where the parameter $t$ determines how strict the criteria should be for determining if the measurements of two vectors are different. This is reflected in the fact that for $t_1 \leq 0 \leq t_2$ we have $\cl{S}^{t_2}_{X,Y} \subset \cl{S}_{X,Y} \subset \cl{S}^{t_1}_{X,Y}$.

The addition of this extra parameter lets us formulate a type of continuity for $d_\cl{P}^t(X,Y)$ where both $t$ and the projections $X$ and $Y$ are allowed to vary. If we want to perturb the projections $X,Y$ by a small amount in operator norm, then we can make up for it by slightly increasing/decreasing the parameter $t$. 

\begin{prop}\label{prop:tSoftOpNormCont}
    Let $\cl{P} = \{P_j\}_{j=1}^m$ be a sequence of projections in $\Proj_\F(n, 2n)$, $t \in \R$, $\epsilon > 0$, and $X_0, Y_0, X, Y \in \Proj_\F(1, 2n)$ such that $\norm{X - X_0} < \epsilon$ and $\norm{Y - Y_0} < \epsilon$. Then
    \begin{equation}
        d_\cl{P}^{t + \epsilon}(X, Y) \leq d_\cl{P}^t(X_0, Y_0) \leq d_\cl{P}^{t - \epsilon}(X,Y)
    \end{equation}
\end{prop}
\begin{proof}
    Suppose $P \in \cl{S}_{X,Y}^{t + \epsilon}$. Then, without loss of generality, we may assume that
    \begin{equation*}
        \tr{PY} + t + \epsilon < \half < \tr{PX} - t - \epsilon.
    \end{equation*}
    
    Since $P$ is a projection we have $\abs{\tr{P(Y_0 - Y)}} \leq \norm{Y - Y_0} < \epsilon$, so
    \begin{equation*}
        \tr{PY_0} + t = \tr{PY} - \tr{P(Y - Y_0)} + t \leq \tr{PY} + t + \epsilon < \half
    \end{equation*}
    and also
    \begin{equation*}
        \tr{PX_0} - t = \tr{PX} - \tr{P(X - X_0)} - t \geq \tr{PX} - t - \epsilon > \half.
    \end{equation*}
    Thus $\cl{S}_{X,Y}^{t + \epsilon} \subset \cl{S}_{X_0,Y_0}^{t}$, and so $d_\cl{P}^{t + \epsilon}(X,Y) \leq d_\cl{P}^t(X_0, Y_0)$.

    The second inequality follows from above by swapping the roles of $X,Y$ with $X_0, Y_0$ and replacing $t$ with $t - \epsilon$.
\end{proof}

\subsubsection{Concentration of $t$-soft Hamming distance}
In this section, we state a basic concentration result for for the $t$-soft Hamming distance between two fixed vectors, and then extend it to a uniform result over an $\epsilon$-net. 

\begin{lem}\label{lem:tSoftPointwiseConc}
    Let $\cl{P} = \{P_j\}_{j = 1}^m$ be an independent sequence of uniformly distributed projections in $\Proj_\F(n,2n)$, $t \in \R$, $\delta > 0$, and $X,Y \in \Proj_\F(1, 2n)$ be fixed. Then
    \begin{equation}
        \Prob{\abs{d_{\cl{P}}^t(X, Y) - \EV{d_\cl{P}^t(X, Y)}} > \delta} \leq 2\Exp{-2\delta^2m}.
    \end{equation}
\end{lem}
\begin{proof}
    From the way that we defined the $t$-soft Hamming distance, $m\cdot d_\cl{P}^t(X, Y) \sim \text{Bin}(m, p)$ where $p = \EV{d_\cl{P}^t(X, Y)}$. The result then follows from a standard Chernoff bound for binomial random variables (see \cite{AlonSpencer:2016}).
\end{proof}

We can now use Proposition~\ref{lem:tSoftPointwiseConc} and the bounds on the size of $\epsilon$-nets of $\Proj_\F(1, 2n)$ from Lemma~\ref{lem:ProjNetSize} to take a union bound. The result is a bound for the probability that the $t$-soft Hamming distance is close to its expectation for all pairs of projections in an $\epsilon$-net simultaneously.

\begin{prop}\label{prop:tSoftNetConc}
    Let $\epsilon > 0$ and $\cl{N}_\epsilon$ be an $\epsilon$-net of $\Proj_\F(1, 2n)$ such that $\log\abs{\cl{N}_\epsilon} \leq 4\beta n\log(1 + 2\epsilon^{-1})$. Also, let $t \in \R$, $\delta > 0$, $m \geq \half\delta^{-2}\p{8\beta n\log(1 + 2\epsilon^{-1}) + D}$, and $\cl{P} = \{P_j\}_{j=1}^m$ be an independent sequence of uniformly distributed projections in $\Proj_\F(n, 2n)$. Then with probability at least $1 - \Exp{-D}$ we have
    \begin{equation}
        \abs{d_\cl{P}^t(X, Y) - \EV{d_\cl{P}^t(X, Y)}} \leq \delta
    \end{equation}
    for all $X,Y \in \cl{N}_\epsilon$.
\end{prop}
\begin{proof}
    By Proposition~\ref{lem:tSoftPointwiseConc} and taking a union bound over all $\binom{\abs{\cl{N}_\epsilon}}{2} \leq \half\abs{\cl{N}_\epsilon}^2$ pairs in $\cl{N}_\epsilon \times \cl{N}_\epsilon$, we have that
    \begin{equation*}
        \Prob{\abs{d_{\cl{P}}^t(X, Y) - \EV{d_{\cl{P}}^t(X, Y)}} \leq  \delta,\ \forall (X,Y) \in \cl{N}_\epsilon \times \cl{N}_\epsilon} \geq 1 - \abs{\cl{N}_{\epsilon}}^2\Exp{-2\delta^2m}.
    \end{equation*}
    Using our bound on the cardinality of $\abs{N_\epsilon}$ and our assumption about $m$ we have
    \begin{equation*}
        \abs{N_{\epsilon}}^2\Exp{-2\delta^2m} \leq \Exp{8\beta n\log(1 + 2\epsilon^{-1}) - 2\delta^2m} = \Exp{-D}.
    \end{equation*}
\end{proof}

The following proposition addresses how varying $t$ affects the expected difference of the $t$-soft Hamming distance from the measurement Hamming distance.

\begin{prop}\label{prop:tSoftExpDiff}
    Let $\cl{P} = \{P_j\}_{j = 1}^m$ be an independent sequence of uniformly distributed projections in $\Proj_\F(n, 2n)$, $t \in \R$, and $X, Y \in \Proj(1, 2n)$ be fixed. Then
    \begin{equation}
        \abs{\EV{d_\cl{P}^t(X, Y) - d_\cl{P}(X, Y)}} = \abs{\Prob{P_1 \in \cl{S}^t_{X,Y}} - \Prob{P_1 \in \cl{S}_{X,Y}}} \leq \frac{32\sqrt{2\beta n - 1}}{e\sqrt{2\pi}}\abs{t}.
    \end{equation}
\end{prop}
\begin{proof}
    Because the $t$-soft and regular Hamming distances are linear combinations of indicator functions, and the fact that the $P_j$ are i.i.d., we have
    
    \begin{equation*}
        \abs{\EV{d_\cl{P}^t(X, Y) - d_\cl{P}(X, Y)}} = \abs{\EV{\ind_{\cl{S}^t_{X,Y}}(P_1) - \ind_{\cl{S}_{X,Y}}(P_1)}},
    \end{equation*}
    and by Jensen's inequality it follows that
    \begin{equation}\label{eq:ProbTSep1}
        \abs{\EV{\ind_{\cl{S}^t_{X,Y}}(P_1) - \ind_{\cl{S}_{X,Y}}(P_1)}} \leq \EV{\abs{\ind_{\cl{S}^t_{X,Y}}(P_1) - \ind_{\cl{S}_{X,Y}}(P_1)}} = \Prob{P_1 \in \cl{S}^t_{X,Y} \triangle \cl{S}_{X,Y}}.
    \end{equation}

    We break up this symmetric difference into two disjoint pieces
    \begin{equation*}
        \Prob{P_1 \in \cl{S}_{X, Y}^t \triangle \cl{S}_{X, Y}} = \Prob{P_1 \in \cl{S}_{X,Y}^t\setminus \cl{S}_{X,Y}} + \Prob{P_1 \in \cl{S}_{X,Y}\setminus \cl{S}_{X,Y}^t}
    \end{equation*}
    and look at two cases. First, if $t > 0$ then $\cl{S}_{X,Y}^t\setminus \cl{S}_{X,Y}$ is empty, and
    \begin{equation*}
        \cl{S}_{X,Y}\setminus \cl{S}_{X,Y}^t \subset \left\{\abs{\tr{P_1X} - \half} < t \right\} \bigcup \left\{\abs{\tr{P_1Y} - \half} < t \right\}.
    \end{equation*}
    Similarly, if $t < 0$ then $\cl{S}_{X,Y}\setminus \cl{S}_{X,Y}^t$ is empty and again
    \begin{equation*}
        \cl{S}_{X,Y}^t\setminus \cl{S}_{X,Y} \subset \left\{\abs{\tr{P_1X} - \half} < -t \right\} \bigcup \left\{\abs{\tr{P_1Y} - \half} < -t \right\},
    \end{equation*}
    Since $\tr{P_1X} \overset{(d)}{=} \tr{P_1Y}$, in both cases we have
    \begin{equation}\label{eq:ProbTSep2}
        \Prob{P_1 \in \cl{S}_{X, Y}^t \triangle \cl{S}_{X, Y}} \leq 2\Prob{\abs{\tr{P_1X} - \half} < \abs{t}}.
    \end{equation}
    
    By Lemma~\ref{lem:DistributionOfMeas} we know $\tr{P_1X} \sim \BetaDist(\beta n, \beta n)$, and so we can bound this probability using the the probability density function of the beta distribution. To begin with, we see
    \begin{align}\label{eq:BetaProbBound}
        \Prob{\abs{\tr{P_1X} - \half} < \abs{t}} &= \frac{2}{B(\beta n, \beta n)} \int_{\half}^{\half + \abs{t}}x^{\beta n - 1}(1 - x)^{\beta n - 1}\ dx\\
        &= \frac{2}{B(\beta n, \beta n)} \int_{0}^{\abs{t}} \p{\frac{1}{4} - x^2}^{\beta n - 1}\ dx \nonumber\\
        &\leq \frac{2}{B(\beta n, \beta n)}\int_{0}^{\abs{t}}4^{1 - \beta n}\ dx \nonumber\\
        &= \frac{8\abs{t}}{4^{\beta n}B(\beta n, \beta n)}. \nonumber
    \end{align}
    Using the lower bound for the Beta function in (\ref{eq:BetaBound}) then yields
    \begin{equation}\label{eq:ProbTSep3}
        \Prob{\abs{\tr{P_1X} - \half} < \abs{t}} \leq \frac{16\sqrt{2\beta n - 1}}{e\sqrt{2\pi}}\abs{t}.
    \end{equation}
    The result follows from combining equation~(\ref{eq:ProbTSep1}) with inequalities~(\ref{eq:ProbTSep2}) and (\ref{eq:ProbTSep3}).
\end{proof}

\subsubsection{Uniform concentration of Hamming distance}\label{Sec-UnifConcHammDist}

We now have all the tools we need to prove that with sufficiently many measurements the Hamming distance concentrates near its expected value for all pairs in $\Proj_\F(1, 2n)$.

\begin{thm}\label{thm:UniformConcHammingDist}
    Let $\delta > 0$, $m \geq 2\delta^{-2}\p{8\beta n\log\p{1 + 128\frac{\sqrt{2 \beta n - 1}}{2\sqrt{2\pi}}\delta^{-1}} + \log(2) + D}$, and $\cl{P} = \{P_j\}_{j=1}^m$ be a collection of independent uniformly distributed projections in $\Proj_\F(n, 2n)$. Then with probability at least $1 - \Exp{-D}$ we have
    \begin{equation}
        \abs{d_\cl{P}(X,Y) - \EV{d_\cl{P}(X,Y)}} < \delta
    \end{equation}
    for all $X,Y \in \Proj_\F(1, 2n)$.
\end{thm}
\begin{proof}
    Let $\epsilon =\frac{e\sqrt{2\pi}}{64\sqrt{2\beta n - 1}}\delta$ and let $\cl{N}_\epsilon$ be an $\epsilon$-net of $\Proj_\F(1, 2n)$ with $\log\abs{\cl{N}_\epsilon} \leq 4\beta n\log(1 + 2\epsilon^{-1})$ as in Lemma~\ref{lem:ProjNetSize}. By our assumption on $m$, Proposition~\ref{prop:tSoftNetConc} says that
    \begin{equation*}
        \Prob{\abs{d_\cl{P}^{\epsilon}(X,Y) - \EV{d_\cl{P}^{\epsilon}(X,Y)}} > \frac{\delta}{2} \text{ for some } X,Y \in \cl{N}_\epsilon} \leq \Exp{-\log(2) - D}
    \end{equation*}
    and also
    \begin{equation*}
        \Prob{\abs{d_\cl{P}^{-\epsilon}(X,Y) - \EV{d_\cl{P}^{-\epsilon}(X,Y)}} > \frac{\delta}{2} \text{ for some } X,Y \in \cl{N}_\epsilon} \leq \Exp{-\log(2) - D},
    \end{equation*}
    and so with probability at least $1 - \Exp{-D}$ we have $\abs{d_\cl{P}^{\pm\epsilon}(X,Y) - \EV{d_\cl{P}^{\pm\epsilon}(X,Y)}} \leq \delta$ for all $X, Y \in \cl{N}_\epsilon$ (call this event $\cl{A})$.
    
    Suppose that $\cl{A}$ occurs. Consider an arbitrary pair $X, Y \in \Proj(1, 2n)$ and let $X_0, Y_0 \in \cl{N}_\epsilon$ such that $\norm{X - X_0} < \epsilon$ and $\norm{Y - Y_0} < \epsilon$. By Proposition~\ref{prop:tSoftOpNormCont} we know that $d_\cl{P}(X, Y) \leq d_\cl{P}^{\epsilon}(X_0,Y_0) \leq d_\cl{P}^{2\epsilon}(X,Y)$. These inequalities together with $\cl{A}$ holding imply
    \begin{equation}
        d_\cl{P}(X, Y) \leq d_\cl{P}^{\epsilon}(X_0, Y_0) \leq \EV{d_\cl{P}^{\epsilon}(X_0, Y_0)} + \frac{\delta}{2} \leq \EV{d_\cl{P}^{2\epsilon}(X, Y)} + \frac{\delta}{2}.
    \end{equation}
    By Proposition~\ref{prop:tSoftExpDiff} we have $\abs{\EV{d_\cl{P}^{2\epsilon}(X,Y)} - \EV{d_\cl{P}(X,Y)}} \leq \frac{32\sqrt{2\beta n - 1}}{e\sqrt{2\pi}}\abs{\epsilon} = \frac{\delta}{2}$, hence
    \begin{equation}
        d_\cl{P}(X,Y) \leq \EV{d_\cl{P}(X, Y)} + \delta.
    \end{equation}
    
    Similarly, using Proposition~\ref{prop:tSoftOpNormCont} again shows that $d_\cl{P}(X,Y) \geq d_\cl{P}^{-\epsilon}(X_0, Y_0) \geq d_\cl{P}^{-2\epsilon}(X,Y)$, and since $\cl{A}$ holds we have
    \begin{equation}
        d_\cl{P}(X, Y) \geq d_\cl{P}^{-\epsilon}(X_0, Y_0) \geq \EV{d_\cl{P}^{-\epsilon}(X_0, Y_0)} - \frac{\delta}{2} \geq \EV{d_\cl{P}^{-2\epsilon}(X, Y)} - \frac{\delta}{2}.
    \end{equation}
    Using Proposition \ref{prop:tSoftExpDiff} as above but for $t = -\epsilon$ yields
    \begin{equation}
        d_\cl{P}(X,Y) \geq \EV{d_\cl{P}(X, Y)} - \delta.
    \end{equation}
\end{proof}

We have just shown that when the measurement projections are chosen uniformly and independently, then $d_\cl{P}(X, Y)$ concentrates near $\EV{d_\cl{P}(X,Y)} = \Prob{P \in \cl{S}_{X,Y}}$ for all $X,Y \in \Proj_\F(1, 2n)$, where $P$ is a single uniformly distributed projection in $\Proj_\F(n, 2n)$. When $n = 1$, then $\Prob{P \in \cl{S}_{X,Y}} = \frac{2}{\pi}\theta \leq \sin(\theta) = \norm{X - Y}$, where $\theta$ is the principal angle between $\Ran(X)$ and $\Ran(Y)$. In the remainder of section, we show that this upper bound holds for arbitrary $n$, see Proposition~\ref{prop:ProbSepBound}. To achieve this, we need to investigate the joint distribution of $(\tr{PX}, \tr{PY})$.


By rotational invariance of the distribution of $P$ we may assume that $\Ran(X)$ and $\Ran(Y)$ are in the two-dimensional subspace spanned by $e_1$ and $e_2$, the first two standard basis vectors. Viewed as matrices, this means that all entries of $X$ and $Y$ are zero outside of the top-left $2 \times 2$ submatrix. Furthermore, if $\tilde{P}, \tilde{X},$ and $ \tilde{Y}$ are the top-left $2 \times 2$ submatrices of their respective matrices then $(\tr{PX}, \tr{PY}) = (\tr{\tilde{P}\tilde{X}}, \tr{\tilde{P}\tilde{Y}})$. We study the joint distribution of $(\tr{PX}, \tr{PY})$ through the submatrix $\tilde{P}$ acting on $\F^2$.

Since $P$ is Hermitian, so is $\tilde{P}$. Thus we may write $\tilde{P} = \lambda_1 E_1 + \lambda_2 E_2$ where $\lambda_1 \geq \lambda_2$ are the eigenvalues of $\tilde{P}$ and $E_1 \perp E_2$ are the projections onto their corresponding eigenspaces. We write $\lambda(\tilde{P}) := (\lambda_1, \lambda_2)$, and $E(\tilde{P}) := (E_1, E_2)$. By the rotational invariance of $P$, $E_1$ is uniformly distributed in $\Proj_2(1, 2)$ and $E_2 = I - E_1$ since Hermitian matrices have mutually orthogonal eigenspaces. Note also that $\lambda(\tilde{P})$ and $E(\tilde{P})$ are independent of each other. The distribution of $\lambda(\tilde{P})$ is given in the following lemma.

\begin{lem}\label{lem:EigenDistNorm}
    Let $n \geq 2$ and $P \in \Proj_\F(n, 2n)$ be uniformly distributed. Then $\lambda(\tilde{P})$ has probability density function $p_n$ on $\cl{D} := \{(x, y) \in [0,1]^2 : y \leq x\}$ defined by
    \begin{equation}
        p_n(x, y) := M_n^{-1} (x - y)^{2\beta}\pb{x(1-x)y(1-y)}^{\beta(n-1) - 1},
    \end{equation}
    with the normalization constant
      \begin{equation}
        M_n = 
        \begin{cases}
            \frac{2}{n-1}B(n-1,n-1) & \text{ if } \F = \R \\
            & \\
            \frac{1}{8n - 4}B(n-1,n-1)^2 & \text{ if } \F = \C.
        \end{cases}
    \end{equation}
\end{lem}
\begin{proof}
    The probability density functions are given by \cite[Proposition 4.1.4]{AndersonGuionnetEtAl:2010} with $p = 2$, $q = 2n-2$, $r = n-2$ and $s = n-2$. It only remains to compute the normalization constants $M_n$.

    Suppose $\F = \R$. Then $p_n(x,y) = M_n^{-1} (x - y)\pb{x(1-x)y(1-y)}^\frac{n-3}{2}$. Define the functions
    \begin{align}
        &f_n(x,y) = -\frac{1}{n-1}\pb{x(1-x)}^{\frac{n-3}{2}}\pb{y(1-y)}^{\frac{n-1}{2}}\\
        &g_n(x,y) = -\frac{1}{n-1}\pb{x(1-x)}^{\frac{n-1}{2}}\pb{y(1-y)}^{\frac{n-3}{2}}.
    \end{align}
    With these definitions, we have $p_n = M_n^{-1} (\frac{\partial g_n}{\partial x} - \frac{\partial f_n}{\partial y})$ on $\cl{D}$. So by Green's theorem,
    \begin{equation}
        1 = \iint_\cl{D} p_n(x, y)\ dxdy = M_n^{-1}\oint_{\partial \cl{D}} f_ndx + g_ndy,
    \end{equation}
    where $\partial \cl{D}$ is the boundary of $\cl{D}$. Note that $f_n$ and $g_n$ both vanish on the boundary of $\cl{D}$ except for the diagonal $\Delta := \{(x,y) \in \cl{D} : x = y\}$, so we only need to compute the line integral over $\Delta$. Parameterizing $\Delta$ by $x(t) = y(t) = 1-t$ for $t \in [0,1]$, we see
    \begin{align}
        M_n= \oint_{\partial \cl{D}} f_n dx + g_n dy &= -\int_0^1 \pb{f_n(x(t), y(t)) + g_n(x(t), y(t))} dt\\
        &= \frac{2}{n-1}\int_0^1t^{n-2}(1-t)^{n-2}dt \nonumber\\
        &= \frac{2}{n-1}B(n-1, n-1). \nonumber
    \end{align}
    
    Next, we consider the case when $\F = \C$. Then $p_n(x,y) = M_n^{-1} (x - y)^2\pb{x(1-x)y(1-y)}^{n-2}$. By symmetry, 
    $1 = \half \iint_{[0,1]^2} p_n(x,y) dxdy$, so by expanding this integral and facts about the Beta distribution, we see
    \begin{equation}
       1= \half \iint_{[0,1]^2} p_n(x,y) dxdy = M^{-1}_n \var(b)\cdot B(n-1, n-1)^2,
    \end{equation}
    where $b \sim \BetaDist(n-1, n-1)$. This beta-distributed random variable has variance $\var(b) = \frac{1}{4(2n-1)}$, which determines $M_n$.
\end{proof}

Let $\cl{D}_{\text{Sep}} := \{(x, y) \in \cl{D} : y < \half < x\}$. Then $\lambda(\tilde{P}) \in \cl{D}_{\text{Sep}}$ if and only if there exist projections $A, B \in \Proj_\F(1, 2)$ such that $\tilde{P} \in \cl{S}_{A, B}$. This is true because $\lambda_1 = \max_{A' \in \Proj(1, 2)}\tr{PA'}$ and $\lambda_2 = \max_{B' \in \Proj(1, 2n)}\tr{PB'}$. In particular, $P \in \cl{S}_{X,Y}$ requires $\lambda(\tilde{P}) \in \cl{D}_{\text{Sep}}$. For this reason, we compute the probability that $\lambda(\tilde{P}) \in \cl{D}_{\text{Sep}}$.

\begin{lem}\label{lem:ProbSep}
    Let $n \geq 2$, and $P \in \Proj_\F(n, 2n)$ be uniformly distributed, then
    \begin{equation}
        \Prob{\lambda(\tilde{P}) \in \cl{D}_{\text{Sep}}} = 
        \begin{cases}
            \frac{B\p{\frac{n-1}{2}, \frac{n-1}{2}}}{2^nB(n-1, n-1)} \to \frac{1}{\sqrt{2}} & \text{ if } \F = \R \\
            & \\
            \frac{1}{2} + \frac{8n - 4}{(n-1)^2 2^{4n - 3}B(n-1, n-1)^2} \to \half + \frac{1}{\pi} & \text{ if } \F = \C.
        \end{cases}
    \end{equation}
    
\end{lem}
\begin{proof}
    First, suppose $\F = \R$, so $p_n(x, y) = M_n^{-1} (x - y)\pb{x(1-x)y(1-y)}^\frac{n-3}{2}$. Then,
    \begin{equation}
        \Prob{\lambda(\tilde{P}) \in \cl{D}_{\text{Sep}}} = M_n^{-1} \int_0^\half \int_\half^1 (x - y)\pb{x(1-x)y(1-y)}^\frac{n-3}{2}dxdy.
    \end{equation}
    By linearity and Fubini's theorem, we get
    \begin{equation*}
        \int_0^\half \int_\half^1 (x - y)\pb{x(1-x)y(1-y)}^\frac{n-3}{2}dxdy = \frac{1}{4}\pb{\EV{b\ \given\ b \geq \half} - \EV{b\ \given\ b \leq \half}}B\p{\frac{n-1}{2}, \frac{n-1}{2}}^2,
    \end{equation*}
    where $b \sim \BetaDist\p{\frac{n-1}{2}, \frac{n-1}{2}}$. Calculating these conditional expectations we get
    \begin{equation*}
        \EV{b\ \given\ b \geq \half} - \EV{b\ \given\ b \leq \half} = \frac{1}{(n-1)2^{n-3}B\p{\frac{n-1}{2}, \frac{n-1}{2}}},
    \end{equation*}
    and combining this with Lemma~\ref{lem:EigenDistNorm} yields
    \begin{equation}
        \Prob{\lambda(\tilde{P}) \in \cl{D}_{\text{Sep}}} = \frac{B\p{\frac{n-1}{2}, \frac{n-1}{2}}}{2^nB(n-1, n-1)}.
    \end{equation}
    
    Next, suppose $\F = \C$, so $p_n(x,y) = M_n^{-1} (x - y)^2\pb{x(1-x)y(1-y)}^{n-2}$. Then,
    \begin{equation}
        \Prob{\lambda(\tilde{P}) \in \cl{D}_{\text{Sep}}} = M_n^{-1}\int_0^\half \int_\half^1 (x - y)^2\pb{x(1-x)y(1-y)}^{n-2}dxdy.
    \end{equation}
    Expanding $(x - y)^2$ and rewriting integrals in terms of expectations of beta-distributed random variables, we see
    \begin{equation}
        \int_0^\half \int_\half^1 (x - y)^2\pb{x(1-x)y(1-y)}^{n-2}dxdy = \frac{1}{2}\pb{\EV{b^2} - \EV{b\ \given\ b \geq \half} \cdot \EV{b\ \given\ b \leq \half}}B(n-1, n-1)^2,
    \end{equation}  
    where $b \sim \BetaDist(n-1, n-1)$. We know that $\EV{b^2} = \frac{n}{4n - 2} = \frac{1}{4} + \frac{1}{8n - 4}$, and also
    \begin{align*}
        \EV{b\ \given\ b \geq \half} \cdot \EV{b\ \given\ b \leq \half} &= \p{\half + \frac{1}{(n-1)2^{2n-2}B(n-1, n-1)}}\p{\half - \frac{1}{(n-1)2^{2n-2}B(n-1, n-1)}}\\
        &= \frac{1}{4} - \frac{1}{(n-1)^2 2^{4n - 4}B(n-1, n-1)^2}. \nonumber
    \end{align*}
    Putting this all together yields
    \begin{equation}
        \Prob{\lambda(\tilde{P}) \in \cl{D}_{\text{Sep}}} = \frac{1}{2} + \frac{8n - 4}{(n-1)^2 2^{4n - 3}B(n-1, n-1)^2}.
    \end{equation}
    
    The asymptotic limit of $\Prob{\lambda(\tilde{P})}$ as $n \to \infty$ follows from Stirling's approximation as in (\ref{eq:BetaBound}), see \cite{Robbins:1955}.
 \end{proof}

Now we are prepared to bound $\Prob{P \in \cl{S}_{X,Y}}$ in terms of the operator norm distance $\norm{X - Y}$.

\begin{prop}\label{prop:ProbSepBound}
    Let $P \in \Proj_\F(n, 2n)$ be uniformly distributed, then 
    \begin{equation}
        \Prob{P \in \cl{S}_{X,Y}} \leq \norm{X - Y}.
    \end{equation}
\end{prop}
\begin{proof}
    The case when $n = 1$ is simple and was mentioned previously, so we consider here $n \geq 2$. Further, without loss of generality, assume $\Ran(X), \Ran(Y) \subset \Ran(E)$ where $E$ is the orthogonal projection onto $\vspan\{e_1,e_2\}$. By conditioning, $\Prob{P \in \cl{S}_{X,Y}} = \EV{\Prob{P \in \cl{S}_{X,Y}\ \given\ \lambda(\tilde{P})}}$. By the definition of $\cl{D}_{\text{Sep}}$ we see that $\Prob{P \in \cl{S}_{X,Y}\ \given\ \lambda(\tilde{P})} = 0$ if $\lambda(\tilde{P}) \in \cl{D}_{\text{Sep}}^c$. Hence
    \begin{equation}
        \EV{\Prob{P \in \cl{S}_{X,Y}\ \given\ \lambda(\tilde{P})}} = \EV{\Prob{P \in \cl{S}_{X,Y}\ \given\ \lambda(\tilde{P})}\ind_{\cl{D}_{\text{Sep}}}(\lambda(\tilde{P}))}.
    \end{equation}
    
    Suppose now that $\lambda(\tilde{P}) \in \cl{D}_{\text{Sep}}$, and first consider the case when $\F = \R$. Then $\Proj_\R(1, 2)$ can be viewed as $\S_\R^1$ with its opposite points identified, and $E(\tilde{P})$ is a (uniformly distributed) random pair of antipodal points in this quotient space. Letting $E_1 = v_1v_1^*$ and $E_2 = v_2v_2^*$
    where $v_1$ and $v_2$ are normalized eigenvectors corresponding to eigenvalues $\lambda_1$ and $\lambda_2$ of $\tilde P$, we may parameterize $\Proj_\R(1, 2)$ by $\phi \in [-\frac{\pi}{2}, \frac{\pi}{2}]$ via $\phi \mapsto Z_\phi := (\cos(\phi)v_1 + \sin(\phi)v_2)(\cos(\phi)v_1 + \sin(\phi)v_2)^* = \cos^2(\phi)E_1 + \sin^2(\phi)E_2 + \sin(\phi)(\cos(\phi)(v_1v_2^* + v_2v_1^*)$. We see that $\tr{\tilde{P}Z_\phi} = \lambda_1\cos^2(\phi) + \lambda_2\sin^2(\phi) = \lambda_1 - (\lambda_1 - \lambda_2)\sin^2(\phi)$. Since $\tr{\tilde{P}Z_0} = \lambda_1 > \half$ and $\tr{\tilde{P}Z_{\frac{\pi}{2}}} = \lambda_2 < \half$, there exists some $\phi_h \in (0, \frac{\pi}{2})$ such that $\tr{\tilde{P}Z_{\phi_h}} = \tr{\tilde{P}Z_{-\phi_h}} = \half$. In fact, $\phi_h = \arcsin\p{\sqrt{\frac{\lambda_1 - \half}{\lambda_1 -\lambda_2}}}$. We see that $\tr{\tilde{P}Z_\phi} > \half$ for $\phi \in (-\phi_h, \phi_h)$, and $\tr{\tilde{P}Z_\phi} < \half$ for $\phi \in [-\frac{\pi}{2}, -\phi_h)\cup(\phi_h, \frac{\pi}{2}]$. 
    
    All of this goes to show that $\lambda(\tilde{P})$ determines $\phi_h$, which along with the orientation of $E_1$ determines which rank-1 projections in $\Ran(E)$ that $P$ separates. In the quotient space picture, the open arc between $E_{\phi_h}$ and $E_{-\phi_h}$ containing $E_1$ represents the rank-1 projections with measurements greater than $\half$, and the complementary arc represents those with measurements less than $\half$. Let $w = \min\{2\phi_h, \pi - 2\phi_h\}$, which is the length of the smallest of these two arcs. If $w \leq \theta$, then $\Prob{P \in \cl{S}_{X,Y}\ \given\ \lambda(\tilde{P})} = \frac{2w}{\pi} \leq \frac{2}{\pi}\theta$. If $w > \theta$, then $\Prob{P \in \cl{S}_{X,Y}\ \given\ \lambda(\tilde{P})} = \frac{2\theta}{\pi}$. So
    \begin{align}
        \EV{\Prob{P \in \cl{S}_{X,Y}\ \given\ \lambda(\tilde{P})}\ind_{\cl{D}_{\text{Sep}}}(\lambda(\tilde{P}))} &\leq \EV{\frac{2\theta}{\pi}\ind_{\cl{D}_{\text{Sep}}}(\lambda(\tilde{P}))}\\
        &= \frac{2\theta}{\pi}\Prob{\lambda(\tilde{P}) \in \cl{D}_{\text{Sep}}}\nonumber\\
        &\leq \norm{X - Y}. \nonumber
    \end{align}
    
    Next, we consider the case when $\F = \C$, in which case $\Proj_\C(1, 2)$ can be identified with the Bloch sphere \cite{Bloch:1946}. By rotational invariance, $E(\tilde{P})$ is a pair of (uniformly distributed) antipodal points on the sphere, and $\lambda(\tilde{P})$ determines which pairs of projections are separated by $P$. If $v_1$ and $v_2$ are eigenvectors of $\tilde P$
    as above, $E_1 = v_1v_1^* $ and $E_2 = v_2v_2^*$, and $v_{\phi, \psi} := \cos(\frac{\phi}{2})v_1 + e^{i\psi}\sin(\frac{\phi}{2})v_2$ for $\phi \in [0, \pi], \psi \in [0, 2\pi]$, then $Z_{\phi, \psi}=v_{\phi, \psi} v_{\phi, \psi}^*$ lies on the circle of points in the Bloch sphere at an angle of $\phi$ from $E_1$. Moreover, this representation shows that $\tr{\tilde{P}Z_{\phi, \psi_1}} = \tr{\tilde{P}Z_{\phi, \psi_2}}$ for all $\phi, \psi_1$, and $\psi_2$. By continuity, there must exist some $\phi_h \in [0, \pi]$ such that $\tr{\tilde{P}Z_{\phi_h, \psi}} = \half$ for all $\psi \in [0, 2\pi]$. In fact, we can calculate $\phi_h = 2\arcsin(\sqrt{\frac{\lambda_1 - \half}{\lambda_1 - \lambda_2}})$. The open spherical cap centered at $E_1$ of angle $\phi_h$ consists exactly of those projections $Z \in \Proj_\C(1, 2)$ such that $\tr{\tilde{P}Z} > \half$, and the complementary cap consists of those for which $\tr{\tilde{P}Z} < \half$.
    
    \begin{figure}[t]
        \centering
        \includegraphics[width=.45\textwidth]{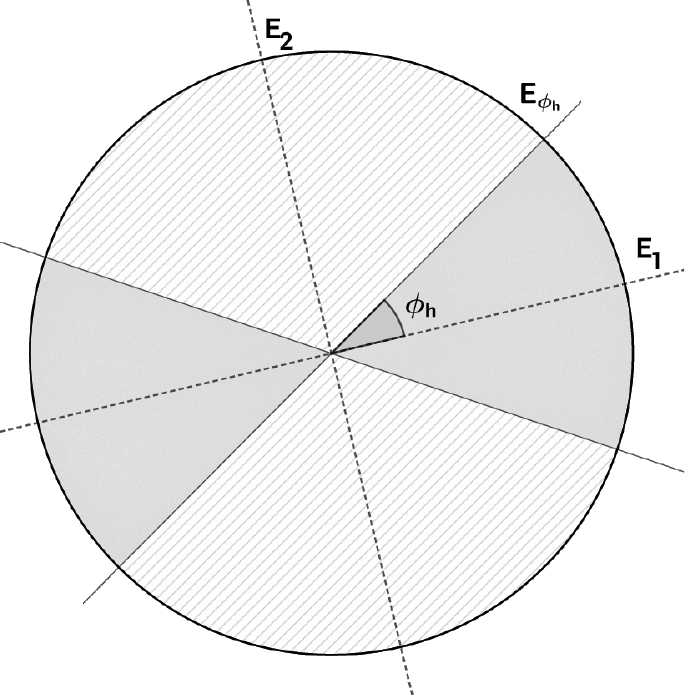} \includegraphics[width=.45\textwidth]{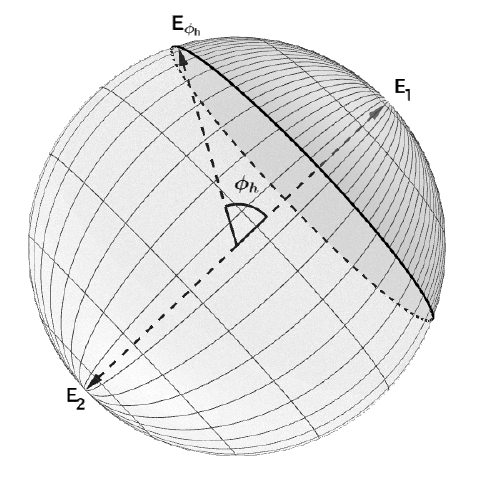}
        \caption{The $2\times 2$ principal submatrix $\tilde{P}$ of $P$ divides $\Proj_\F(1, 2)$ into two disjoint sets based on whether the Hilbert-Schmidt inner product of a rank-one orthogonal projection with $\tilde{P}$ is greater or less than $\half$ (Left: $\F = \R$; Right: $\F = \C$). If $P$ separates two points $X$ and $Y$, then $\tilde P =\lambda_1 E_1 + \lambda_2 E_2$ with
        eigenvalues $\lambda_1>1/2>\lambda_2$ and mutually orthogonal eigenprojectors $E_1$ and $E_2$. The subset shaded in darker gray contains the points for which the Hilbert-Schmidt inner product with $\tilde P$ is greater than 1/2.}
    \end{figure}
    
    Conditioning on $\lambda(\tilde{P})$ determines the opening angles of these two spherical caps, which are oriented along a random diameter determined by $E(\tilde{P})$. The projections $X, Y$ are two fixed points on the Bloch sphere at an angle of $2\theta$, and are separated if and only if they are not in the same cap. Let $w = \min\{\phi_h, \pi - \phi_h\}$, which is the smallest opening angle of these two caps. If $w \leq \theta$, then any cap of angle $w$ containing $X$ cannot contain $Y$ (and vice versa), so $\Prob{P \in \cl{S}_{X,Y}\ \given\ \lambda(\tilde{P})}$ is just twice the normalized area of a cap of angle $w$ (which is just its normalized height), i.e.
    \begin{equation}
        \Prob{P \in \cl{S}_{X,Y}\ \given\ \lambda(\tilde{P})} = 1 - \cos(w) \leq 1 - \cos(\theta) \leq \sin(\theta) = \norm{X - Y}.
    \end{equation}
    If $w > \theta$, then it is possible for both $X$ and $Y$ to be in a cap of opening angle $w$. In this case, $\Prob{P \in \cl{S}_{X,Y}\ \given\ \lambda(\tilde{P})}$ is just the normalized area of the symmetric difference of spherical caps of angle $w$ centered at $X$ and $Y$. The intersection of these two caps contains a spherical cap of angle $w - \theta$ centered at the geodesic midpoint of $X$ and $Y$, so for this case
    \begin{align}
        \Prob{P \in \cl{S}_{X,Y}\ \given\ \lambda(\tilde{P})} \leq \cos(w - \theta) - \cos(w) \leq \sin(\theta) = \norm{X - Y}.
    \end{align}
    where the last inequality follows since $w \leq \frac{\pi}{2}$. Thus we have
    \begin{align}
        \EV{\Prob{P \in \cl{S}_{X,Y}\ \given\ \lambda(\tilde{P})}\ind_{\cl{D}_{\text{Sep}}}(\lambda(\tilde{P}))} &\leq \norm{X - Y}\Prob{\lambda(\tilde{P}) \in \cl{D}_{\text{Sep}}}\\
        &\leq \norm{X - Y}. \nonumber
    \end{align}

\end{proof}

The uniform bound for the measurement Hamming distance in terms of the operator norm distance now follows directly by combining Theorem~\ref{thm:UniformConcHammingDist} with Proposition~\ref{prop:ProbSepBound}.

\begin{thm}\label{thm:UnifHammingOpIneq}
Let $\delta > 0$, $m \geq 2\delta^{-2}\p{8\beta n\log\p{1 + 128\frac{\sqrt{2 \beta n - 1}}{2\sqrt{2\pi}}\delta^{-1}} + \log(2) + D}$, and $\cl{P} = \{P_j\}_{j=1}^m$ be a collection of independent uniformly distributed projections in $\Proj_\F(n,2n)$. Then with probability at least $1 - \Exp{-D}$
\begin{equation}
    d_\cl{P}(X, Y) \leq \norm{X - Y} + \delta
\end{equation}
for all $X, Y \in \Proj_\F(1, 2n)$.
\end{thm}

\subsection{Uniform guarantees for accurate recovery}
With the results from Sections \ref{subsec:ConcEmpAvgNet} and \ref{subsec:SoftHammingArgument} we are  ready to extend the pointwise result given in Theorem~\ref{thm:Pointwise} to a uniform result that controls the behavior of our recovery procedure for all input vectors simultaneously. 

\begin{thm}\label{thm:Uniform}
    Let $\delta > 0$ and set $\epsilon = \frac{(\mu_1 - \mu_2)}{8}\delta$. If 
    \begin{equation}\label{eq:mUniform}
        m \geq 2\epsilon^{-2}\p{8\beta n\log\p{1 + 128\frac{\sqrt{2 \beta n - 1}}{2\sqrt{2\pi}}\epsilon^{-1}} + 2\log(2) + D}
    \end{equation}
    and $\cl{P} = \{P_j\}_{j=1}^m$ is an independent sequence of uniformly distributed projections in $\Proj_\F(n, 2n)$, then with probability at least $1 - \Exp{-D}$
    \begin{equation}
        \norm{\hat{X} - X} < \delta
    \end{equation}
    for all $X \in \Proj(1, 2n)$, where $\hat{X}$ is the solution to (PEP) with input $\Phi_\cl{P}(X)$.
\end{thm}
\begin{proof}
    Let $\cl{N}_\epsilon$ be an $\epsilon$-net for $\Proj_\F(1, 2n)$ such that $\log\abs{\cl{N}_\epsilon} \leq 4\beta n\log(1 + 2\epsilon^{-1})$ as in Lemma~\ref{lem:ProjNetSize}. By our choice of $m$, Lemma~\ref{lem:EmpAvgConcNet} says that with probability greater than $1 - \Exp{-\log(2)-D}$ we have $\norm{\hat{Q}_\cl{P}(X) - Q(X)} \leq \epsilon$ for all $X \in \cl{N}_\epsilon$ (call this event $\cl{A}$). Also by our choice of $m$, Theorem~\ref{thm:UnifHammingOpIneq} says that with probability at least $1 - \Exp{-\log(2) - D}$ we have $d_\cl{P}(X,Y) \leq \norm{X - Y} + \epsilon$ for all $X,Y \in \Proj_\F(1, 2n)$ (call this event $\cl{B}$).
    
    Suppose that $\cl{A}$ and $\cl{B}$ both occur, which happens with probability at least $1 - \Exp{-D}$, and consider an arbitrary $X \in \Proj_\F(1, 2n)$. We know from (\ref{eq:ErrorOpNormIneq}) that
    \begin{equation}\label{eq:ErrorIneq}
        \norm{\hat{X} - X} \leq 2(\mu_1 - \mu_2)^{-1}\norm{\hat{Q}_\cl{P}(X) - Q(X)}.
    \end{equation}

    To bound the right-hand side of this last inequality we pass to the $\epsilon$-net $\cl{N}_\epsilon$ by picking $X_0 \in \cl{N}_\epsilon$ with $\norm{X - X_0} < \epsilon$. Then
    \begin{equation}\label{eq:TriangleIneq}
        \norm{\hat{Q}_\cl{P}(X) - Q(X)} \leq \norm{\hat{Q}_\cl{P}(X) - \hat{Q}_\cl{P}(X_0)} + \norm{\hat{Q}_\cl{P}(X_0) - Q(X_0)} + \norm{Q(X_0) - Q(X)}.
    \end{equation}
    
    Next, we examine each of the three terms on the right side of (1\ref{eq:TriangleIneq}). To bound the first term, note that $\abs{\ps{j : \hat{P}_j(X) \neq \hat{P}_j(X_0)}} = m\cdot d_\cl{P}(X, X_0)$. Using this and the assumption that $\cl{A}$ holds yields
    \begin{equation}
        \norm{\hat{Q}_\cl{P}(X) - \hat{Q}_\cl{P}(X_0)} = \norm{\frac{1}{m}\sum_{j: \hat{P}_j(X) \neq \hat{P}_j(X_0)} \hat{P}_j(X) - \hat{P}_j(X_0)} \leq d_\cl{P}(X,X_0) \leq 2 \epsilon.
    \end{equation}
    Since $\cl{B}$ holds, we can bound the second term by $\norm{\hat{Q}_\cl{P}(X_0) - Q(X_0)} \leq \epsilon$. Lastly, using Proposition~\ref{prop:ExpFlippedProjs} gives $Q(X) - Q(X_0) = (\mu_1 - \mu_2)(X - X_0)$, and so we can bound the third term by
    \begin{equation}
        \norm{Q(X_0) - Q(X)} = (\mu_1 - \mu_2)\norm{X - X_0} \leq (\mu_1 - \mu_2)\epsilon.
    \end{equation}
    
    Using these three bounds together in (\ref{eq:TriangleIneq}) gives
    \begin{equation}
        \norm{\hat{Q}_\cl{P}(X) - Q(X)} \leq 3\epsilon + (\mu_1 - \mu_2)\epsilon \leq \half(\mu_1 - \mu_2)\delta,
    \end{equation}
    which combined with (\ref{eq:ErrorIneq}) yields $\norm{\hat{X} - X} < \delta$.
\end{proof}

\begin{figure}
    \centering
    \includegraphics[width=0.8\textwidth]{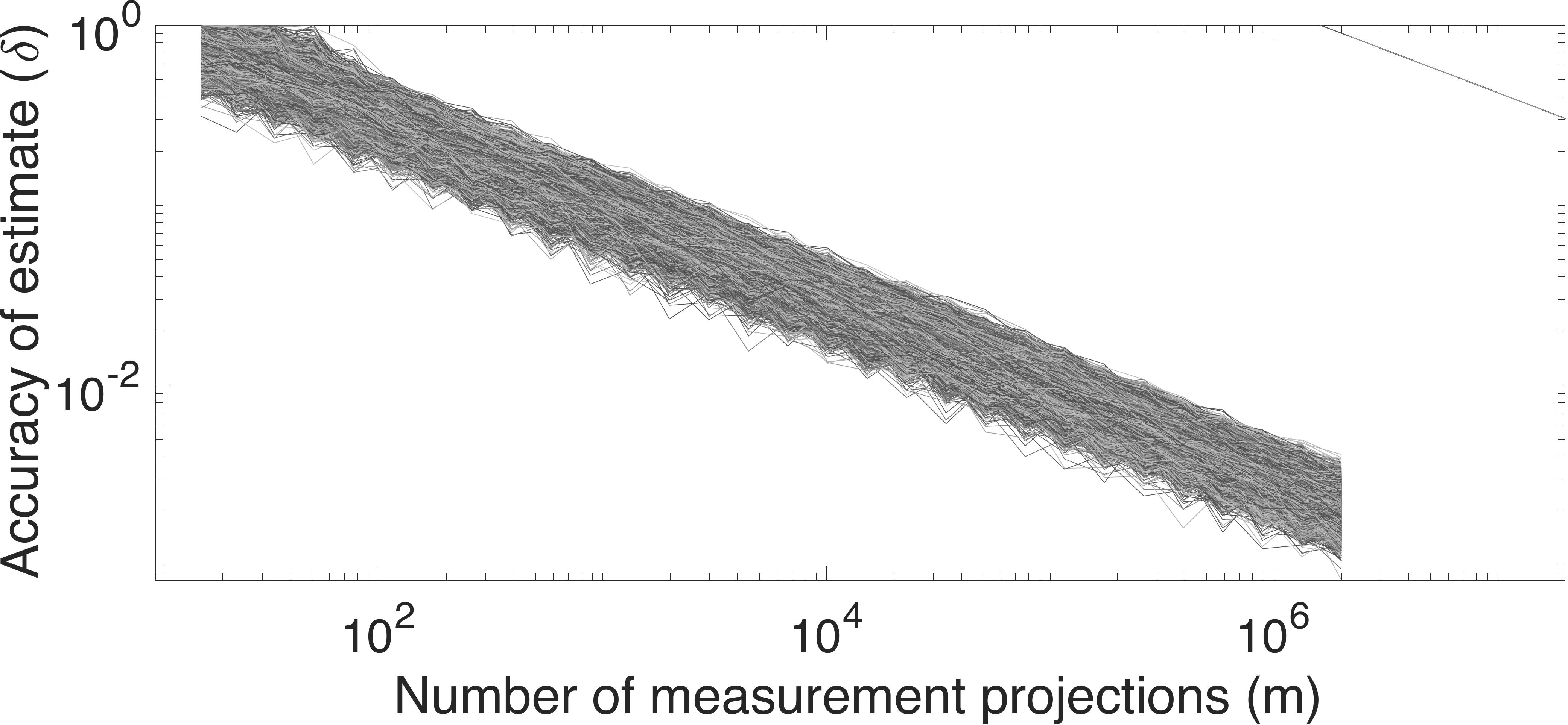}
    \caption{Plot showing the accuracy for the recovery of 15000 random inputs using (PEP) with a fixed collection of measurement projections on $\R^{16}$. The single line separate from the cluster represents the upper bound on $\delta$ given by Theorem~\ref{thm:Uniform}.}
    \label{fig:UniformPlot}
\end{figure}

See Figure~\ref{fig:UniformPlot} for a plot showing how our bound on the sufficient number of measurements to achieve a uniform accuracy of $\delta$ relates to experimental results.

As in the pointwise case, our proof allows us to fine tune the probability that a sequence of measurement projections provides uniformly accurate recovery by adjusting the value of $D$ in (\ref{eq:mUniform}). In particular, we can take $D = n$ to ensure success with overwhelming probability, i.e. the failure rate decays exponentially in $n$. In the pointwise case, this resulted in gaining an additional factor of $n$ in the number of measurement projections, see Corollary~\ref{cor:PointwiseOverwhelming}. In the uniform case, however, the asymptotics remain the same.

\begin{cor}
    Let $\delta > 0$. If $\alpha > 0$, $m \geq C \delta^{-2}n^2\log(\delta^{-1}n)$, and $\cl{P} = \{P_j\}_{j=1}^m$ is an independent sequence of uniformly distributed projections in $\Proj_\F(n, 2n)$, then with probability at least $1 - \Exp{-n}$
    \begin{equation}
        \norm{\hat{X} - X} < \delta
    \end{equation}
    for all $X \in \Proj(1, 2n)$, where $\hat{X}$ is the solution to (PEP) with input $\Phi_\cl{P}(X)$ and $C$ is a constant.
\end{cor}

As in the pointwise case, we can modify the proof of Theorem~\ref{thm:Uniform} to show that uniform recovery using (PEP) is robust to bit-flip errors occurring in a faulty measurement
$\tilde{\Phi}_{\cl P}$.

\begin{cor}
    Let $\delta$, $m$, and $\{P_j\}_{j=1}^m$ be as in Theorem~\ref{thm:Uniform}, and additionally let $0 < \tau < 1$. Then with probability at least $1 - \Exp{-D}$, for all $X \in \Proj_\F(1, 2n)$ and all $\tilde{\Phi}_\cl{P}(X) \in \{0,1\}^m$ with
    \begin{equation}
        d_H(\Phi_\cl{P}(X), \tilde{\Phi}_\cl{P}(X)) \leq \tau
    \end{equation}
    we have
    \begin{equation}
        \norm{\tilde{X} - X} \leq \delta + 2(\mu_1 - \mu_2)^{-1}\tau,
    \end{equation}
    where $\tilde{X}$ is the solution to (PEP) with input $\tilde{\Phi}_\cl{P}(X)$ and $\mu_1-\mu_2$ is bounded by Lemma~\ref{lem:SpecGapBound}.
\end{cor}
\begin{proof}
    Let $\tilde{Q}_\cl{P}(X)$ denote the empirical average of the (faulty) flipped projections, i.e. flipped using $\tilde{\Phi}_\cl{P}(X)$ rather than $\Phi_\cl{P}(X)$. As before, for all $X \in \Proj_\F(1, 2n)$ we have
    \begin{equation}
        \norm{\tilde{X} - X} \leq 2(\mu_1 - \mu_2)^{-1}\norm{\tilde{Q}_\cl{P}(X) - Q(X)}.
    \end{equation}
    Using the triangle inequality, we expand
    \begin{equation}
        \norm{\tilde{Q}_\cl{P}(X) - Q(X)} \leq \norm{\tilde{Q}_\cl{P}(X) - \hat{Q}_\cl{P}(X)} + \norm{\hat{Q}_\cl{P}(X) - Q(X)} \leq \tau + \norm{\hat{Q}_\cl{P}(X) - Q(X)}.
    \end{equation}
    Bounding $\norm{\hat{Q}_\cl{P}(X) - Q(X)}$ with high probability proceeds exactly as in Theorem~\ref{thm:Uniform}.
\end{proof}


\end{document}